\documentclass{nm}

\renewcommand\thisnumber{x}
\renewcommand\thisyear {2023}
\renewcommand\thismonth{x}
\renewcommand\thisvolume{xx}
\renewcommand\doi{10.4208/nmtma. }

\usepackage[utf8]{inputenc}
\usepackage{geometry}
\usepackage{xcolor,soul}
\usepackage{float}
\usepackage{booktabs}
\usepackage{mathtools}
\usepackage{url}
\usepackage{bm}
\usepackage{tikz}
\usepackage{subcaption}
\usepackage{amsmath}
\usepackage{amsthm}

\usepackage{amssymb}
\usepackage{graphicx}
\usepackage{setspace}
\theoremstyle{plain}

\newtheorem{rem}{Remark}
\newtheorem{lem}{Lemma}
\newtheorem{thm}[lem]{Theorem}
\newtheorem{prop}[lem]{Proposition}
\newtheorem{cor}[lem]{Corollary}
\newtheorem{assum}{Assumption}

\usepackage{algorithm}
\usepackage[noend]{algpseudocode}

\usepackage{pgfkeys} 	%
\usepackage{ifthen} 		%

\pgfkeys{
/tensor/.cd, %
dim1/.initial = 1,		dim1/.get = \dimOne,		dim1/.store in = \dimOne,
dim2/.initial = 1,		dim2/.get = \dimTwo,		dim2/.store in = \dimTwo,
dim3/.initial = 1,		dim3/.get = \dimThree,	dim3/.store in = \dimThree,
xshift/.initial = 0, 	xshift/.get = \xShift, 		xshift/.store in = \xShift,
yshift/.initial = 0, 	yshift/.get = \yShift, 		yshift/.store in = \yShift,
xspec/.initial = 0, 	xspec/.get = \xSpec, 		xspec/.store in = \xSpec,
yspec/.initial = 0, 	yspec/.get = \ySpec, 		yspec/.store in = \ySpec,
scale/.initial = 1, 	scale/.get = \myScale, 	scale/.store in = \myScale,
fill/.initial = white, 	fill/.get = \myFill, 		fill/.store in = \myFill,
back edges/.initial = 0, 	back edges/.get = \myBack, 	back edges/.store in = \myBack,
slice type/.initial = none, 		slice type/.get = \sliceType, 			slice type/.store in = \sliceType, %
number of slices/.initial = 1, 	number of slices/.get = \nSlices, 		number of slices/.store in = \nSlices,
slice width/.initial = 1, 		slice width/.get = \sWidth, 				slice width/.store in = \sWidth, %
}

\newcommand{\tensor}[1][]{\@tensor[#1]}
\def\@tensor[#1] (#2,#3) #4; {{ %

\pgfkeys{/tensor/.cd,#1}

\def\depthScale{0.5} %

\pgfmathsetmacro{\numSlicesMinusOne}{\nSlices-1}
\pgfmathsetmacro{\numSlicesPlusOne}{\nSlices+1}

\pgfmathsetmacro{\sliceLength}{\myScale*\dimOne}

\ifthenelse{\equal{\sliceType}{lateral}}
	{
	
	\pgfmathsetmacro{\sliceWidth}{\myScale*\sWidth*0.9*\dimTwo/\nSlices}
	\pgfmathsetmacro{\sliceGap}{\myScale*\dimTwo/(\nSlices-1) - \nSlices*\sliceWidth/(\nSlices-1)}
	\pgfmathsetmacro{\sliceDepth}{\myScale*\dimThree}
	
	} 
	{
	\ifthenelse{\equal{\sliceType}{frontal}}
		{
		
		\pgfmathsetmacro{\sliceDepth}{\myScale*\sWidth*0.9*\dimThree/\nSlices}
		\pgfmathsetmacro{\sliceGap}{\myScale*\dimThree/(\nSlices-1) - \nSlices*\sliceDepth/(\nSlices-1)}
		\pgfmathsetmacro{\sliceWidth}{\myScale*\dimTwo}
	
		}
		{
		\pgfmathsetmacro{\sliceWidth}{\myScale*\dimTwo}
		\pgfmathsetmacro{\sliceDepth}{\myScale*\dimThree}
		}

	}

\def\xFront{#2 + \xShift}	
\def\yFront{#3 + \yShift}
\def\xBack{#2 + \xShift + \depthScale*\sliceDepth + \xSpec*\sliceDepth}
\def\yBack{#3 + \yShift + \depthScale*\sliceDepth + \ySpec*\sliceDepth}

\def\aFront{(\xFront, \yFront)}
\def\bFront{(\xFront, \yFront + \sliceLength)}
\def\cFront{(\xFront + \sliceWidth, \yFront + \sliceLength)}
\def\dFront{(\xFront + \sliceWidth, \yFront)}

\def\aBack{(\xBack, \yBack)}
\def\bBack{(\xBack, \yBack + \sliceLength)}
\def\cBack{(\xBack + \sliceWidth, \yBack + \sliceLength)}
\def\dBack{(\xBack+ \sliceWidth, \yBack)}

\ifthenelse{\NOT\equal{\myFill}{nofill}}
	{
	\def\tempTensor{
		\fill[\myFill!25] \bFront -- \bBack -- \cBack -- \cFront -- cycle; %
		\fill[\myFill!75] \dFront -- \dBack -- \cBack -- \cFront -- cycle; %
		\fill[\myFill!50] \aFront rectangle \cFront;  %
	
		\draw \aFront rectangle \cFront; %
		\draw \bFront -- \bBack; 
		\draw \cFront -- \cBack;
		\draw \dFront -- \dBack;
	
		\draw \bBack -- \cBack;
		\draw \cBack -- \dBack;
		}
	}
	{ %

	\def\tempTensor{
		\draw \aFront rectangle \cFront; %
		
		\ifthenelse{\NOT\equal{\myBack}{0}}
		{
			\draw[dashed] \bBack -- \aBack -- \dBack;
		}{}
		
		\draw \dBack -- \cBack -- \bBack;

		\ifthenelse{\NOT\equal{\myBack}{0}}
		{
			\draw[dashed] \aFront -- \aBack;
		}{}
		
		\draw \bFront -- \bBack;
		\draw \cFront -- \cBack;
		\draw \dFront -- \dBack;
		}
	}

\ifthenelse{\equal{\sliceType}{lateral}}
	{
	\foreach\sliceCount in {0,...,\numSlicesMinusOne}
		{	
		\begin{scope}[shift ={(\sliceCount*\sliceWidth + \sliceCount*\sliceGap, 0)}]
			\tempTensor;
		\end{scope}
		}
	
	}
	{
	
	\ifthenelse{\equal{\sliceType}{frontal}}
	{
	
	\pgfmathsetmacro{\xStep}{\sliceDepth/2 + \sliceGap/2 + \myScale*\dimThree*\xSpec/(\nSlices-(1-\sWidth))}
	\pgfmathsetmacro{\yStep}{\sliceDepth/2 + \sliceGap/2 +  \myScale*\dimThree*\ySpec/(\nSlices-(1-\sWidth))}
	
	\foreach\sliceCount in {-\numSlicesMinusOne,...,0}
		{	
		
		\begin{scope}[shift = {(-\sliceCount*\xStep, -\sliceCount*\yStep)}]
			\tempTensor;
		\end{scope}
	
		}
	
	}
	{
	\tempTensor;
	}
	
	}

\node at (#2 + \dimTwo/2, #3 + \dimOne/2) {#4};

}} %

\newcommand{\am}[1]{\textcolor{orange}{\textit{#1}}}

\newcommand{\Mod}[1]{\ (\mathrm{mod}\ #1)}

\newcommand{\fold}[1]{\text{fold}\left(#1\right)}
\newcommand{\unfold}[1]{\text{unfold}\left(#1\right)}

\newcommand{\bcirc}[1]{\text{bcirc}\left(#1\right)}

\providecommand{\keywords}[1]{{\textit{Keywords:}} #1}

\begin{document}

\markboth{    }{      }
\title{Frontal Slice Approaches for Tensor Linear Systems}

\author[ ]{ Hengrui Luo\affil{1}\affil{3},  Anna Ma\affil{2} }
\address{
\affilnum{1}\  Department of Statistics, Rice University, Houston, 77005, USA\\
\affilnum{2}\ Department of Mathematics, University of California, Irvine, 92617, USA \\
\affilnum{3}\  Computational Research Division, Lawrence Berkeley National Laboratory, Berkeley, 94701, USA\\
[2ex]
\rm }%

\emails{
{\tt  } ( ),
{\tt  } ( )
}

\begin{abstract}

Inspired by the row and column action methods for solving large-scale linear systems, in this work, we explore the use of frontal slices for solving tensor linear systems. In particular, this paper presents a novel approach for using frontal slices of a tensor $\mathcal{A}$ to solve tensor linear systems $\mathcal{A} * \mathcal{X} = \mathcal{B}$ where $*$ denotes the t-product. In addition, we consider variations of this method, including cyclic, block, and randomized approaches, each designed to optimize performance in different operational contexts. Our primary contribution lies in the development and convergence analysis of these methods. Experimental results on synthetically generated and real-world data, including applications such as image and video deblurring, demonstrate the efficacy of our proposed approaches and validate our theoretical findings.

\end{abstract}


\maketitle

\section{Introduction}
\thispagestyle{plain}
In the realm of contemporary data science, the availability of multi-dimensional
data, commonly referred to as tensors, has catalyzed transformative
advances across diverse domains such as machine learning \cite{cho2023surrogate,reichel2022tensor},
neuroimaging \cite{lyu2023bayesian}, recommendation systems \cite{chi2020provable} and signal processing \cite{murray2023randomized,sidiropoulos2017tensor,chi2012tensors}.
Tensors, which extend beyond the simpler constructs of matrices, encapsulate
higher-order interactions within data that matrices alone cannot.
While potentially offering a more comprehensive framework for analysis
and predictive modeling \cite{kolda2009tensor,chi2012tensors}, tensors come with
complexity and high dimensionality, which introduce escalated computational
costs and demanding storage requirements, particularly in large-scale and high-fidelity compressible 
datasets \cite{bengua2017efficient, kielstra2023tensor}. 

In this work, we are interested in solving large-scale consistent tensor
multi-linear systems of the form 
\begin{equation}
\mathcal{A}*\mathcal{X}=\mathcal{B},\label{eq:tensorlinsys}
\end{equation}
where $\mathcal{A}\in\mathbb{R}^{n_{1}\times n_{2}\times n}$, $\mathcal{X}\in\mathbb{R}^{n_{2}\times n_{3}\times n}$,
$\mathcal{B}\in\mathbb{R}^{n_{1}\times n_{3}\times n}$, and $*$
denotes the tensor product, known as the \textit{t-product} \cite{kilmer2011factorization}. Direct solvers for solving tensor linear systems induce high computational complexity \cite{grasedyck2004existence,chen2024low,chen2021regularized}. 
As an alternative, iterative algorithms for such approximating solutions to \eqref{eq:tensorlinsys}
have been previously studied \cite{liang2019alternating,dehdezi2022rapid,li2022gradient}
and have applications for image and video deblurring, regression, dictionary
learning, and facial recognition \cite{hao2013facial,el2022tensor,reichel2022tensor}.

The t-product can be viewed as a generalization of the matrix-vector product. In particular, when $n=1$ and $n_{3}=1$, the t-product simplifies to the matrix-vector
product. In the matrix-data setting, row and column iterative methods
have been proposed to solve large-scale linear systems of equations
\begin{equation}
\bm{A}\bm{x}=\bm{b},\label{eq:matrixlinsys}
\end{equation}
where $\bm{A}\in\mathbb{R}^{n_{1}\times n_{2}},\bm{b}\in\mathbb{R}^{n_{1}\times1}$
are given and $\bm{x}\in\mathbb{R}^{n_{2}\times1}$ is unknown. When $n_{1}$ and $n_{2}$ are very large, solving
the linear system directly (e.g., by computing the pseudo-inverse)
quickly becomes impractical. In other large-scale settings, one may
not even be able to load all entries but only a few rows or columns
of the matrix $\bm{A}$ at a time. In such settings, stochastic iterative
methods with low memory footprints, such as the Randomized Kaczmarz
or Randomized Gauss-Seidel (RGS) algorithms, can be used to solve \eqref{eq:matrixlinsys}.
The relationship between these two methods, one using rows of the matrix and the other using columns of the matrix, has been studied in previous
works \cite{ma2015convergence}. Such row and column action methods solving linear systems have been further generalized to a framework
known as \textit{sketch-and-project} \cite{gower2015randomized,raskutti2016statistical,gower2021adaptive,murray2023randomized,cho2023surrogate}.

Sketching simplifies computations by solving a sub-system as a proxy to the original system while preserving the data's
intrinsic characteristics, thereby addressing the practical limitations
of direct manipulation due to size or complexity \cite{raskutti2016statistical,tropp2017practical, murray2023randomized}, which proves essential in scenarios where handling full datasets 
is impractical. A primary advantage of sketching is the enhancement
of computational efficiency. In tensor operations such as multiplications
or factorizations, multiplication complexity (and so is the solving complexity) can increase exponentially with the
sizes of the tensor mode (See Proposition \ref{prop:The-computational-complexity}). Sketching techniques reduce the effective
dimensionality of the tensor system, which in turn accelerates computations
and diminishes the burden on memory resources. This improvement is
vital in real-time processing applications like video processing \cite{LHL2024, liu2022tensor,  kolda2009tensor},
where swift and efficient processing is crucial.

Previous works considered sketches under the t-product. For example, \cite{ma2022randomized} propose the
tensor Randomized Kaczmarz Algorithm, where ``row-slices''
(instead of matrix rows), corresponding to row sketching, are utilized at every iteration (TRK). Other works
\cite{murray2023randomized} have
considered solving tensor multi-linear systems using a general sketch-and-project
framework. The work by  \cite{ma2022randomized} encapsulates the TRK algorithm for the sketch-and-project regime for tensor systems.

Although the tensorized versions of RK and the sketch-and-project
framework have been proposed and studied, they do not capture settings
unique to the tensor landscape by nature of being extensions of matrix-based
algorithms. In this work, we propose a suite of \textit{frontal-slice-based}
iterative methods for approximating the solution to tensor linear systems.
Our main contributions include: 
\begin{enumerate}
\item Variations (cyclic, block, and randomized) of methods that utilize frontal slices to approximation
solutions to multi-linear tensor systems under the t-product
\item Convergence and computational complexity analysis for the cyclic frontal slice-based method and its variants.  
\item Empirical comparison of our proposed methods to other tensor
sketching methods and real-world image and video data. 
\end{enumerate}
All the aforementioned algorithms are indicative of a broader trend in tensor computation,
where methods are adapted or developed to exploit the multi-dimensional
structure of tensors. By focusing on individual slices, these algorithms offer a more manageable and potentially
more efficient approach to solving tensor-based problems compared
to methods that must consider the entire tensor at each step.

\paragraph{Organization. }

The rest of the paper is organized as follows: Section~\ref{sec:Tensor-t-product}
presents related work and background on the tensor t-product. Section
\ref{sec:Algorithmic-Approaches} covers our main algorithms, including
the classic gradient descent and our proposed algorithms, which balance
storage and computational complexities. Section \ref{sec:Convergence-Analysis}
provides algorithm convergence guarantees. Section \ref{sec:Experiments} uses synthetic
and real-world experiments to validate our results empirically. Lastly, Section
\ref{sec:Future-works} discusses our findings and points to some
interesting future works.

\section{\label{sec:Tensor-t-product}Background}

This section reviews related works and presents necessary background
information regarding the tensor t-product.

\paragraph{Notation.}

Calligraphic letters are reserved for tensors, capital letters for
matrices, and lowercase letters for vectors and indices. We adopt
MATLAB notation to reference elements or slices of a tensor, where
a ``slice" refers to a sub-tensor in which one index is fixed. For
example, $\mathcal{A}_{i::}$ denotes the $i$-th row slice of dimension
$1\times n_{2}\times n$ of $\mathcal{A}$ and $a_{ijk}\in\mathbb{R}$
denotes the $(i,j,k)$-th element of $\mathcal{A}$. Since frontal
slices will be used frequently, we shorten the notation of frontal
slices to $\mathcal{A}_{k}$. Figure~\ref{fig:slices} presents a visual representation of the row and frontal slices of a tensor. We assume that $\mathcal{A}$ and $\mathcal{B}$ in \eqref{eq:tensorlinsys} are known, and our goal is to approximate   $\mathcal{X}$, where $\mathcal{X}$ is the unique, exact solution to \eqref{eq:tensorlinsys}.

\begin{figure}[ht]
    \begin{center}
\begin{tabular}{ccc}
\begin{subfigure}{.45\linewidth}
\centering
  \resizebox{0.4\textwidth}{!}{
    \begin{tikzpicture}
    	\tensor[dim1 = 0.75, dim2 = 1., dim3 = 0.75, fill = lightgray] (6  ,-.9) {};
        \tensor[dim1 = 0.15, dim2 = 1., dim3 = 0.75, fill = lightgray] (6 ,0) {};
    \end{tikzpicture}
  }
\caption{Row slice $\mathcal{A}_{1::}$ of tensor $\mathcal{A}$.}
\label{fig:row_slice}
\end{subfigure}
& 
\begin{subfigure}{.45\linewidth}  
    \centering
  \resizebox{0.4\textwidth}{!}{
    \begin{tikzpicture}	
    	\tensor[dim1 = 1., dim2 = 1., dim3 = .75, fill = lightgray] (10  ,-.9) {};
        \tensor[dim1 = 1., dim2 = 1., dim3 = 0.15, fill = lightgray] (9.8 ,-1.1) {};
    \end{tikzpicture}
    }
    \caption{Frontal slice $\mathcal{A}_{1}$ of tensor $\mathcal{A}$.}
    \label{fig:frontal_slice}
\end{subfigure}
\end{tabular}
\end{center}
\caption{The first row slice $\mathcal{A}_{1::}$ and frontal slice $\mathcal{A}_{1}$ of $\mathcal{A}$.}
\label{fig:slices}
\end{figure}
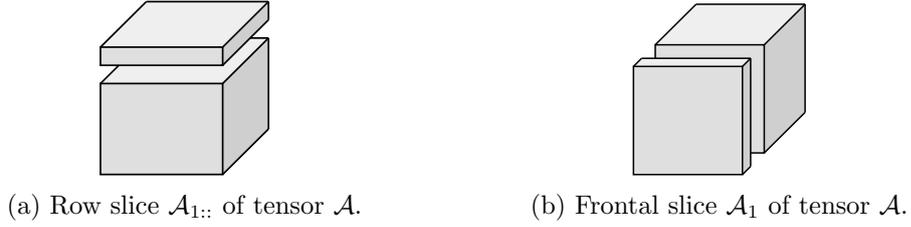

\subsection{The Tensor t-product}

\label{sec:tprod}The tensor t-product \cite{kilmer2011factorization}
is an operation in tensor algebra motivated to generalize the concept of matrix
multiplication to tensors. While the t-product can be further generalized
to higher order tensors and to use other orthogonal transforms \cite{kernfeld2015tensor,song2019robust},
here, we present the original form of the t-product, using third-order
tensors %
associated with the Fourier Transform. 
Readers who are well-acquainted
with the t-product may skip this section and proceed to Section~\ref{sec:Algorithmic-Approaches}. Before defining the t-product, we first introduce two tensor operations. 
\begin{definition}
(Tensor operations, \cite{kilmer2011factorization}) \label{defn:bcirc}For
$\mathcal{A}\in\mathbb{R}^{n_{1}\times n_{2}\times n}$, let $\unfold{\mathcal{A}}$
denote the unfolded tensor 
\[
\unfold{\mathcal{A}}=\begin{pmatrix}A_{1}\\
A_{2}\\
\vdots\\
A_{n}
\end{pmatrix}\in\mathbb{R}^{n_{1}n\times n_{2}}.
\]
The inverse of the unfolding operation is denoted $\fold{\cdot}$
and is such that $\fold{\unfold{\mathcal{A}}}=\mathcal{A}$. Let $\bcirc{\mathcal{A}}$ denote the block-circulant matrix 
\begin{equation}
\bcirc{\mathcal{A}}=\begin{pmatrix}A_{1} & A_{n} & A_{n-1} & \ldots & A_{2}\\
A_{2} & A_{1} & A_{n} & \ldots & A_{3}\\
\vdots & \vdots & \vdots & \ddots & \vdots\\
A_{n} & A_{n-1} & A_{n-2} & \ldots & A_{1}
\end{pmatrix}\in\mathbb{R}^{n_{1}n\times n_{2}n}.\label{eqn:bcirc}
\end{equation}
\end{definition}

\begin{definition}
(Tensor product (t-product), \cite{kilmer2011factorization})
\label{defn:tprod} Given $\mathcal{A}\in\mathbb{R}^{n_{1}\times n_{2}\times n}$ and
$\mathcal{X}\in\mathbb{R}^{n_{2}\times n_{3}\times n}$, the t-product is defined as 
\[
\mathcal{A}*\mathcal{X}=\fold{\bcirc{\mathcal{A}}\unfold{\mathcal{X}}}\in\mathbb{R}^{n_{1}\times n_{3}\times n}.
\] 
\end{definition}
Since circulant matrices are diagonalizable by the Fourier Transform,
computationally, the t-product can be efficiently performed by computing
the Discrete Fourier Transform (DFT) along each tensor's third dimension,
multiplying the corresponding frontal slices in the Fourier domain
and then applying the Inverse Discrete Fourier Transform (IDFT) to
the product. It should be noted that because this transformation is applied along the third mode of the tensor, it requires all $n$ frontal slices.  So if only one frontal slice of $\mathcal{A}$ in \eqref{eq:tensorlinsys} is given at a time, the $\bcirc{\mathcal{A}}$ is not complete and its diagonalization is not immediately computable.

\paragraph{Data perspectives. }
Unlike matrix-vector products, data tensors are flexible that they admits a variety of tensor products. Before we move on, let us briefly compare the t-product to two other popular tensor products: the Kronecker product
and the Khatri-Rao product. Practitioners usually treat matrices as data whose rows are samples and
columns are features. In the tensor data setting \cite{bi2021tensors,LHL2024},
data tensors use additional dimensions (e.g., frontal dimension) to
represent multiple instances of sample feature sets. One concrete example is video tensor $\mathcal{A}$, where frontal slices $\mathcal{A}_{i}$ represent different pixel-by-pixel frames. 

The Kronecker product between two matrices $\bm{A}$ and $\bm{B}$
results in a new matrix where each scalar element $a_{ij}$ of matrix
$\bm{A}$ is multiplied by the entire matrix $\bm{B}$. 
This can be extended to tensors. In particular, the Kronecker product between
two tensors $\mathcal{A}$ and $\mathcal{B}$ is defined as a larger
tensor $\mathcal{C}$ such that each element $a_{ijk}$ of $\mathcal{A}$
is multiplied by the entire tensor $\mathcal{B}$. Geometrically,
this represents a multidimensional scaling where every element of
$\mathcal{B}$ is scaled by   corresponding elements of $\mathcal{A}$,
creating a much larger tensor that encodes all combinations of the
elements from $\mathcal{A}$ and $\mathcal{B}$.

The Khatri-Rao product \cite{sidiropoulos2017tensor,liu2008hadamard}
is   a ``column-wise'' Kronecker product. For matrices,
it can be thought of as taking the Kronecker product of corresponding
columns from two matrices. %
For two tensors $\mathcal{A}$
and $\mathcal{B}$ with equal numbers of slices along a dimension, the Khatri-Rao product is formed by the
column-wise Kronecker product of corresponding slices from $\mathcal{A}$
and $\mathcal{B}$ along that shared dimension. Geometrically, this
can be viewed as combining features from two different matrices (or
tensors) at a more granular (coarser but more regulated) level than
the Kronecker product. If we imagine each column of the matrix (or slice
of the tensor) as a feature, the Khatri-Rao product combines these features
in a way that every feature in one vector (tensor slice) is paired with
every feature in the corresponding column of the other matrix (tensor slice), producing a detailed and granular new
feature space. It preserves more structure compared to the Kronecker
product and is often used in scenarios where the interaction between
corresponding features (columns) is crucial.

The t-product \cite{kilmer2011factorization} allows tensors to 
interact across slices in a more complex way, using a chosen transform, in our case the Fourier transform.
In contrast to the Kronecker and Khatri-Rao products, the t-product
utilizes a convolutional operation along the   frontal
dimension   of the tensors $\mathcal{A}$ and
$\mathcal{B}$. If we imagine each column of the matrix (or frontal slice
of the tensor) as a feature, the t-product combines every feature in one
vector (or tensor) paired with every feature in the corresponding
column of the other matrix (or slice of the other tensor) by taking
product in the Fourier domain (hence taking convolution in the original
domain) to form a new feature space. This convolution-like process
essentially mixes the frontal slices of the tensors based on their
relative positions, integrating them into a new tensor whose slices
are combinations of the original slices manipulated through the circulant
structure. However, when the frontal dimension is of size 1, there is no convolution, reducing the t-product to our usual matrix
product.

As the need for higher-dimensional data representations grows in fields
such as computer vision, signal processing, and machine learning,
exploring and developing operations that extend
the t-product to higher-order tensors becomes essential. These higher-order tensors
(order greater than three) can represent more complex data structures,
capturing additional modes of variability. The higher-order t-product can be defined recursively, as suggested
by \cite{martin2013order}. For instance, consider a fourth-order
tensor $\mathcal{A}\in\mathbb{R}^{n\times m\times p\times q}$. The
t-product for fourth-order tensors can be constructed by treating
$\mathcal{A}$ as a tensor in $\mathbb{R}^{n\times n\times(pq)}$,
where the last two modes are merged into a single mode, and then applying
the t-product recursively.

In extending linear systems sketching to the tensor setting, there are multiple options for operating on tensors such as using the t-product \cite{wang2015fast,li2017near}. The t-product has been of particular interest
due to its natural extension of many desirable linear algebraic properties
and definitions.

In this work, we focus on t-product for tensors of 3 modes due to
its wide applicability. The tensor t-product
generalizes linear transformations and matrix multiplication to higher
dimensions, facilitating the development of algorithms that can manipulate
multi-dimensional arrays efficiently. The transformation step typically
uses fast transform methods, such as the Fast Fourier Transform (FFT),
which are computationally efficient and widely supported in various
software and hardware implementations, enhancing the practicality
of using tensor t-products in large-scale applications. Unlike matrices,
tensors can encapsulate data with multiple dimensions without losing
inherent multi-way relationships, making tensor operations particularly
suitable for complex data structures found in modern applications
like signal processing and machine learning. In Figure \ref{fig:hst},
we show the frontal sketching effectively deblurs images by treating
them as third-order tensors. The sketching method not only preserves
spatial and temporal correlations, ensuring each update incorporates
information from all slices, but also provides a significantly clearer
image, which demonstrates the power of tensor-based techniques in
image deblurring.
Before presenting our algorithmic approach, we first present additional useful definitions for this work.

\begin{figure}
\centering \includegraphics[width=1\textwidth]{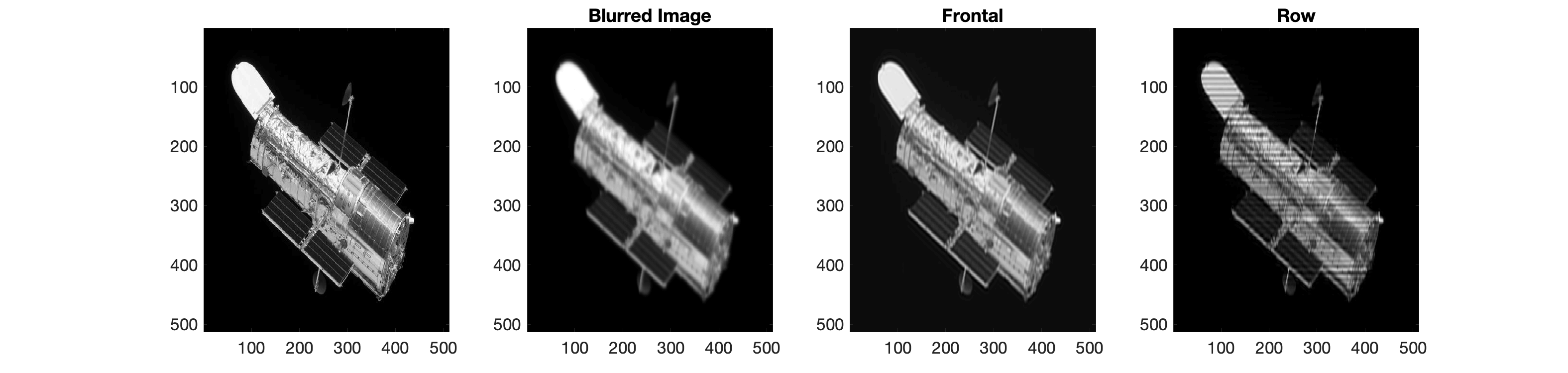}
\caption{The original $512\times512$ pixel image of the Hubble Space Telescope
\protect\protect\protect\protect\protect\url{https://github.com/jnagy1/IRtools/blob/master/Extra/test_data/HSTgray.jpg},
its Gaussian blurred version and the frontal sketching (Algorithm
\ref{alg:tsolve_cyclic_slices}) deblurred version. We defer the details
of this deblurring experiment to Section \ref{sec:Experiments}.}
\label{fig:hst} 
\end{figure}

\begin{definition}
(Identity tensor, \cite{kilmer2011factorization}) \label{defn:identity}
The identity tensor $\mathcal{I}\in\mathbb{R}^{n_{1}\times n_{1}\times n}$
is defined as a tensor whose first frontal slice is the identity matrix
and is zero elsewhere. 
\end{definition}

\begin{definition}
(Transpose, \cite{kilmer2011factorization}) \label{defn:transpose}
Let $\mathcal{A}\in\mathbb{R}^{n_{1}\times n_{2}\times n}$ then $\mathcal{A}^{T}$
is obtained by transposing each frontal slice and reversing the order
of slices $2$ to $n$. 
\end{definition}

Note that the identity tensor is defined such that, under the t-product,
$\mathcal{A}*\mathcal{I}=\mathcal{A}$ and $\mathcal{I}*\mathcal{A}=\mathcal{A}$,
for appropriately sized tensors. Furthermore, the transpose is defined
such that not only does $(\mathcal{A}^{T})^{T}=\mathcal{A}$ hold,
it also holds that under the t-product $(\mathcal{A}*\mathcal{B})^{T}=\mathcal{B}^{T}*\mathcal{A}^{T}$,
which is not generally valid for other tensor products. 
\begin{definition}
(Tensor Norms, \cite{du2021randomized}) \label{defn:norms} We define
the squared Frobenius norm of $\mathcal{A}\in\mathbb{R}^{n_{1}\times n_{2}\times n}$
as 
\begin{equation}
\|\mathcal{A}\|_{F}^{2}=\sum_{i=1}^{n_{1}}\sum_{j=1}^{n_{2}}\sum_{k=1}^{n}a_{ijk}^{2},\label{eq:fronorm}
\end{equation}
where $a_{ijk}$ is the $(i,j,k)$-th entry of $\mathcal{A}$. The
operator norm of $\mathcal{A}$ is defined
as 
\begin{equation}
\|\mathcal{A}\|_{op}\coloneqq\sup_{\|\mathcal{X}\|_{F}=1}\|\mathcal{A}*\mathcal{X}\|_{F}=\|\bcirc{\mathcal{A}}\|_{2},\label{eq:specnorm}
\end{equation}
where $\|\cdot\|_{2}$ denotes the usual matrix 2-norm for the block circulant matrix. 
\end{definition}

\begin{rem}
     by Definition~\ref{defn:norms}, it holds that 
\[
\|\mathcal{A}*\mathcal{X}\|_{F}\leq\|\mathcal{A}\|_{op}\|\mathcal{X}\|_{F},
\]
for appropriately sized tensors $\mathcal{A}$ and $\mathcal{X}$.
\label{rem:op-fro}
\end{rem}

\subsection{Related Works}
Major interest from the statistics community focuses
on extending the tensor regression problem, whose solution depends
on \eqref{eq:tensorlinsys}. For example, \cite{lu2024statistical}
high-dimensional quantile tensor regression utilizes convex decomposable
regularizers to approximate the solution to ~\eqref{eq:tensorlinsys}
using the incremental proximal gradient method, which bears much worse
complexity than gradient descent \cite{bertsekas2011incremental}.
Another generalization by \cite{he2014graphical} incorporates an
$L_{1}$ norm penalization term and solves \eqref{eq:tensorlinsys} using a block coordinate
descent algorithm and each block of coordinates \cite{friedman2008sparse}.
While the convergence rates of the coordinate descent algorithm for tensor
variant is unknown, the nested optimization introduces a higher
complexity then the gradient descent. For generalized tensor models
\cite{ghannam2023tensor}, the optimization problem based on \eqref{eq:tensorlinsys}
can be even more expensive, although it still depends on \eqref{eq:tensorlinsys}
in a vectorized form. The scalability of the gradient descent algorithm
needs to be investigated beyond the small size (i.e., each mode of
the tensor is usually smaller than 100 in these papers' examples)
when the common practice of sketch-and-project is in the play.

Previous works have considered the application of iterative methods to solve ~\eqref{eq:tensorlinsys}. 
In~\cite{ma2022randomized}, the authors propose a tensor-based Kaczmarz method. The proposed method was shown to be a special case of the more general tensor sketch-and-project (TSP) method, which was proposed and studied by~\cite{tang2022sketch}, based on the matrix sketch-and-project framework of~\cite{gower2015randomized}. 
In the TSP method, a sketch tensor $\mathcal{S}\in\mathbb{R}^{k\times n_{1}\times n}$
is multiplied with $\mathcal{A}$ to obtain its sketch, and then an
objective is minimized subject to the sketch itself, i.e., the iterates $\mathcal{X}(t+1)$ are approximated by solving the minimization problem
\begin{equation}
\mathcal{X}(t+1)=\text{argmin}_{\mathcal{X}}\|\mathcal{X}-\mathcal{X}(t)\|_{F}^{2}\text{ s.t. }\mathcal{S}^{T}*\mathcal{A}*\mathcal{X}=\mathcal{S}^{T}*\mathcal{B}.
\end{equation}
The closed-form solution of this minimization problem is written as: 
\begin{equation}
\mathcal{X}(t+1)=\mathcal{X}(t)+\mathcal{A}^{T}*\mathcal{S}*\left(\mathcal{S}^{T}*\mathcal{A}*\mathcal{A}^{T}*\mathcal{S}\right)^{\dagger}*\mathcal{S}*\left(\mathcal{B}-\mathcal{A}*\mathcal{X}(t)\right),\label{eq:TSP_update}
\end{equation}
where the notation $\left(\cdot\right)^{\dagger}$ denotes the Moore-Penrose
pseudoinverse defined in~\cite{tang2022sketch}. 
Choosing $k=1$
and $\mathcal{S}\in\mathbb{R}^{1\times n_{2}\times n}$ where, for
a fixed index $i\in[n_{2}]$, $\mathcal{S}_{1ij}=1$ for all $j\in[n]$
and zero otherwise, creates a ``row sketch'' of $\mathcal{A}$,
i.e., $\mathcal{S}_{i}*\mathcal{A}=\mathcal{A}_{i::}$.
In this setting, single-row slices of the data
tensor $\mathcal{A}$ and row slices of $\mathcal{B}$ are utilized a time, and the TSP method recovers the TRK algorithm of ~\cite{ma2022randomized}. Similarly, a column-slice-based method, which only uses single column slices at a time, can be derived using the TSP method. Unfortunately, for sketches $\mathcal{S}_i$ to obtain frontal slices $\tilde{\mathcal{A}}_{i}$, where $\tilde{\mathcal{A}}_{i} \in \mathbb{R}^{n_1 \times n_2 \times n}$ denotes a tensor with $\mathcal{A}_{i}$ in the $i$-th frontal slice and zero elsewhere, $\mathcal{S}_i$ would need to depend on $\mathcal{A}$, using all frontal slices of $\mathcal{A}$. This may not be feasible if only frontal slice-wise information is accessible at a time, thus restracting the use of sketch and project methods to reduce computational overhead.

While column-slice and row-slice sketching are well-studied, the sketching
along the frontal dimension, which differentiates tensors from matrices,
remains unexplored. 
In this work, we consider a unique
setting to solving tensor systems via frontal-slice sampling.
While row and column sketching fall under the sketch-and-project framework
\cite{wilkinson_distance-preserving_2022,cho2023surrogate,murray2023randomized}, the approaches proposed here does not. 

\section{\label{sec:Algorithmic-Approaches}The Frontal Slice Descent Algorithm}

We approach the design of an iterative method that uses frontal slices of $\mathcal{A}$ via a decomposition of the gradient descent iterates applied to the least squares objective. In doing so, we derive two variations of the proposed method, which we refer to as the Frontal Slice Descent (FSD) method. 
To that end, let $\mathcal{F}(\mathcal{X})=\frac{1}{2} \|\mathcal{A}*\mathcal{X}-\mathcal{B}\|^{2}$
denote the least squares objective function for the system \eqref{eq:tensorlinsys}, then the gradient
can be written as $\nabla\mathcal{F}=\mathcal{A}^{T}*(\mathcal{A}*\mathcal{X}-\mathcal{B})$. Thus, applying gradient descent to approximate solutions to~\eqref{eq:tensorlinsys} produces the iterates
\begin{equation}
    \mathcal{X}(t+1) = \mathcal{X}(t) + \alpha \mathcal{A}^{T} * (\mathcal{B} - \mathcal{A} * \mathcal{X}(t)),
    \label{eq:GD}
\end{equation}
where $\alpha \in \mathbb{R}^{+}$ is assumed to be a fixed learning rate for simplicity. 
Algorithm~\ref{alg:tsolve_full_gradient} presents pseudo-code for this approach. In the algorithm, we inital residual tensor $\mathcal{R}=\mathcal{B} - \mathcal{A} * \mathcal{X}$ is
set to $\mathcal{B}$ and the inital approximation of $\mathcal{X}$ is set to
zero. Then, iterative updates to $\mathcal{R}$ and $\mathcal{X}$ are performed. The update for $\mathcal{R}$ involves
subtracting the product of $\mathcal{A}$ and $\mathcal{X}$, representing
the difference between the current estimate and the measurements.
The update for $\mathcal{X}$ involves a step in the direction of
the gradient of the residual, scaled by a learning rate $\alpha$.
This is akin to classical gradient descent but adapted for tensors.

The use of the full gradient \eqref{eq:SFSD} in each step, suggests
that this algorithm is designed for scenarios where computing the
full gradient is feasible. This may not be the case for very large-scale
problems, where only some of $\mathcal{A}$ is available at a time due to (e.g., memory) constraints. %
This algorithm is the primary variant for its application
in multi-dimensional data contexts, offering a way to harness the
structure and relationships inherent in tensor data. For practitioners,
it's important to consider the computational cost and the convergence
properties of this method, particularly in relation to the size and
complexity of the data tensors involved.

Algorithm \ref{alg:tsolve_full_gradient}
can be computationally expensive, especially for large tensors. We propose a novel algorithm that uses one frontal slice per iteration to remedy the per-iteration information needed from $\mathcal{A}$. %
This adaptation presents an innovative way to exploit
the multi-dimensional nature of tensors, offering a feasible
and efficient approach to tensor-based optimization problems using only frontal slices of $\mathcal{A}$ per iteration.

\begin{algorithm}[h!]
\caption{Gradient descent for t-product least squares problem \eqref{eq:tensorlinsys}.}
\label{alg:tsolve_full_gradient}

\begin{algorithmic}

\State \textbf{Input:} data matrix $\mathcal{A}$, measurements $\mathcal{B}$,
learning rate $\alpha$

\State \textbf{Initialize} $\mathcal{R}(0)=\mathcal{B}$, $\mathcal{X}(0)=\mathbf{0}$,
$t=0$

\While{stopping criteria not satisfied}

\State $\mathcal{R}(t+1)= \mathcal{B} - \mathcal{A}*\mathcal{X}(t) $ \Comment{Residual at iteration $t+1$}

\State $\mathcal{X}(t+1)=\mathcal{X}(t) + \alpha\mathcal{A}^{T}*\mathcal{R}(t+1)$\Comment{Gradient step} 
\State $t=t+1$

\EndWhile

\end{algorithmic} 
\end{algorithm}

\subsection{Using Frontal Slices}

The t-product gradient descent algorithm (Algorithm \ref{alg:tsolve_full_gradient})
geometrically represents an iterative process of aligning tensor slices
through dynamic convolutional interactions. The circular convolution
integrates spatial relationships between slices, and the transpose
operation ensures that adjustments are made coherently to align the
tensors optimally.

In the $t$-th iteration, the algorithm works iteratively to align
the tensor $\mathcal{X}(t)$ with $\mathcal{B}$ in the t-product
sense. Each iteration refines this alignment through the convolutional
interactions represented by the t-product. The circular convolution
$\mathcal{A}*\mathcal{X}(t)$ integrates adjacent slice information,
and the algorithm adjusts $\mathcal{X}(t)$ into $\mathcal{X}(t+1)=\mathcal{X}(t)+\alpha\mathcal{A}^{T}*\mathcal{R}(t+1)$
to account for misalignments across adjacent slices in $\mathcal{B}$.
The gradient $\mathcal{A}^{T}*\mathcal{R}(t+1)$ points in the direction
of the steepest descent, where each iteration accounts for the convolutional
effect of adjacent slices, hence maintaining the geometric integrity
of the circular convolution relationship. Given appropriate step sizes
$\alpha$, the method iteratively reduces the error until it converges
to the solution.%

The proposed method approximates the solution to $\mathcal{A}*\mathcal{X}=\mathcal{B}$
by iteratively updating the tensor $\mathcal{X}$ based on the current
residual $\mathcal{R}$ and the selected slice of $\mathcal{A}$.
Using a single slice for each update, as opposed to the full
tensor can significantly reduce communication cost (needing only $\mathcal{A}_i$ with $n_1n_2$ entries  instead of all $n_1n_2n$ elements from $\mathcal{A}$),  making
these algorithms suitable for large-scale tensor problems. The remainder of this section focuses on the derivation of our method and the residual approximation.

Instead of using all frontal slices in each iteration, our goal is to approximate $\nabla \mathcal{F}$ using only one frontal slice at a time. Although $\mathcal{F}$ cannot be additively decomposed with respect to frontal slices, we can still design methods that use a frontal slice at a time. Define $\tilde{\mathcal{A}}_{i}\in\mathbb{R}^{n_{1}\times n_{2}\times n}$, such that the $i$-th frontal slice of $\tilde{\mathcal{A}}_{i}$ is
the $i$-th frontal slice of $\mathcal{A}$ and $\tilde{\mathcal{A}}_{i}$ is zero elsewhere. Then, $\mathcal{A}$
can be written as a sum of its frontal slices $\mathcal{A} = \sum_{i=1}^{n} \tilde{\mathcal{A}}_{i}$. Thus, the gradient can be written as
$$\nabla\mathcal{F}(\mathcal{X}) = \sum_{i=1}^{n}\tilde{\mathcal{A}}_{i}^{T}*\left( \mathcal{B} - \mathcal{A} * \mathcal{X} \right).$$
Unfortunately, the residual, $\mathcal{B} - \mathcal{A}*\mathcal{X}$ cannot be computed without all frontal slices. However, it can be approximated. Suppose that at iteration $i(t)$, frontal slice $\mathcal{A}_{i(t)}$ is given (e.g.,  $i(t)=i(t-n+1)$ represents cyclic indexing), then we propose to approximate $\nabla \mathcal{F}$ with  
\begin{equation}
    \nabla \tilde{F}_i = \tilde{\mathcal{A}}_{i(t)}^T * \mathcal{R}(t+1),
\end{equation}
where 
\begin{equation}
\mathcal{R}(t+1) =
\begin{cases}
    \mathcal{R}(t) - \tilde{\mathcal{A}}_{i(t)}^T* \mathcal{X}(t) & \text{if } t \leq n \\ 
    \mathcal{R}(t) - \tilde{\mathcal{A}}_{i(t)}^T \mathcal{X}(t) + \tilde{\mathcal{A}}_{i(t)}^T * \mathcal{X}(t-n+1) & \text{otherwise,}
\end{cases}
\label{eq:resid_approx}
\end{equation}
is an approximation of the residual, which only depends on one frontal slice in a fixed iteration. Then the \emph{Frontal Slice Descent} (FSD) algorithm approximates the solution to~\eqref{eq:tensorlinsys} via the iterates
\begin{equation}
\mathcal{X}(t+1)=\mathcal{X}(t)+\alpha\tilde{\mathcal{A}}_{i(t)}^{T}*\mathcal{R}(t+1),\label{eq:SFSD}
\end{equation}
where $\alpha\in\mathbb{R}_{+}$ denotes the learning rate.

The residual approximation step in our proposed method can be interpreted as a delayed computation of the exact residual, whose idea is similar to boosting in modeling literature \cite{LHL2024,friedman2001greedy}. In particular, it can be shown that 
\begin{equation}
    \mathcal{R}(t+1) = \mathcal{B} - \mathcal{A} * \mathcal{X}(t-n+1) + \mathcal{E},
\end{equation}
where $\mathcal{E}$ contains terms depending on the iterates $\mathcal{X}(t-n+2), ..., \mathcal{X}(t)$, with the convention that $\mathcal{X}(t)=0$ for $t\leq0$. Furthermore, when   frontal slices are mutually orthogonal, i.e., $\mathcal{A}_i^T * \mathcal{A}_j = 0$, $\mathcal{R}(t+1) = \mathcal{B} - \mathcal{A} * \mathcal{X}(t)$,   FSD performs similarly to gradient descent.

While the proposed method reduces data demand by removing the requirement of all $n$ frontal slices of $\mathcal{A}$, its drawback is that it requires a historical memory of $n$ versions of $\mathcal{X}$. In practical applications such as image and video deblurring, this drawback can be significantly circumvented due to the structure of $\mathcal{A}$, specifically when only a small number of frontal slices of $\mathcal{A}$ are nonzero. See Section~\ref{sec:Experiments} for more details.

\subsection{\label{subsec:Frontal-Slice-Descent} Variations of Frontal Slice Descent}

In this subsection, we present variations of the proposed approach, including using multiple frontal slices and randomly selecting slices.

\textbf{Cyclic FSD.} 
Our main approach iterates over the frontal slices of the tensor cyclically.
By focusing on one frontal slice at a time, the proposed method simplifies the computation
and reduces the immediate computational load. This cyclic approach
ensures that all slices are treated uniformly over time, providing
a comprehensive coverage of the tensor's dimensions. This method might
be more predictable and methodical, which can be beneficial in certain
applications where systematic use of all data is crucial. The theoretical guarantees presented in Section~\ref{sec:Convergence-Analysis} assume frontal slices are selected cyclically.

\textbf{Block FSD.} The pseudo-code for a generalized version of our proposed FSD is provided in Algorithm~\ref{alg:tsolve_cyclic_slices}. Instead of using single frontal slices, our method can be extended to $s$ frontal slices simultaneously. In particular, let
\begin{equation}
    \tilde{\mathcal{A}}_{i}^{s}=\tilde{\mathcal{A}}_{i\cdot(s-1)+1}+\tilde{\mathcal{A}}_{i\cdot(s-1)+2}+\cdots+\tilde{\mathcal{A}}_{i\cdot s+1},
\end{equation}
denote a frontal sub-block of $\mathcal{A}$. When $s=1$, Algorithm~\ref{alg:tsolve_cyclic_slices} simplifies to the single frontal slice setting. When $s$ is larger, we expect that
the approximation improves (See Appendix \ref{subsec:Blocked-Case}
for a discussion on blocking strategy). 
Algorithm \ref{alg:tsolve_cyclic_slices}
 performs partial gradient updates across slices
until completing a full gradient step, based on considering the decomposition
of the full gradient update $\mathcal{A}^{T}*(\mathcal{A}*\mathcal{X}-\mathcal{B})=\sum_{i=1}^{n}\tilde{\mathcal{A}}_{i}^{T}*(\mathcal{A}*\mathcal{X}-\mathcal{B})=\sum_{i=1}^{n}\tilde{\mathcal{A}}_{i}^{T}*\mathcal{R}$. It obtains more information from $\mathcal{A}$ per iteration compared to the single-slice Algorithm \ref{alg:tsolve_cyclic_slices} and generalizes it. 
We leave the study of optimal blocking strategies for future work.

\begin{algorithm}[t]
\caption{Block Gradient descent with cyclic frontal slices for t-product least squares
problem \eqref{eq:tensorlinsys}.}
\label{alg:tsolve_cyclic_slices} \begin{algorithmic}

\State \textbf{Input:} data matrix $\mathcal{A}$, measurements $\mathcal{B}$,
learning rate $\alpha$, block size $s$

\State \textbf{Initialize} $\mathcal{R}(0)=\mathcal{B}$, $\mathcal{X}(0)=\mathbf{0}$,
$t=0$

\While{stopping criteria not satisfied}

\State $i=\text{mod}(t,n/s)+1$

\State $\mathcal{R}(t+1)=\mathcal{R}(t)-\tilde{\mathcal{A}}_{i}^{s}*\mathcal{X}(t)+\tilde{\mathcal{A}}_{i}^{s}*\mathcal{X}(t-n+1)$ 

\Comment{Residual Approximation with the convention that $\mathcal{X}(t)=0$ for $t\leq0$.}

\State $\mathcal{X}(t+1)=\mathcal{X}(t)+\alpha\left(\tilde{\mathcal{A}}_{i}^{s}\right)^T*\mathcal{R}(t+1)$\Comment{Solution
approximation}

\State $t=t+1$

\EndWhile \end{algorithmic} 
\end{algorithm}

\textbf{Random FSD.}
Algorithm~\ref{alg:tsolve_random_slices} presents a stochastic approach where  frontal slices are selected randomly  in each
iteration. This stochastic element can potentially lead to faster
convergence in some cases. Random selection can also be more effective
in dealing with tensors where certain slices are more informative
than others, as it does not systematically prioritize any particular
order of slices. 

In this case, the memory complexity remains the same, with the requirement of storing $n$ historical approximations $\mathcal{X}(k_i)$ where $k_i$ is the last iteration slice $i$ was selected. The residual approximate then is replaced with,
\begin{align}
    \mathcal{R}(t+1)=\mathcal{R}(t)-\tilde{\mathcal{A}}_{i}*\mathcal{X}(t)+\tilde{\mathcal{A}}_{i}*\mathcal{X}(k_i)\label{eq:resid_approx_random}
\end{align}
with the convention that $\mathcal{X}(t)=0$ for $t\leq0$.

\begin{algorithm}[H]
\caption{Gradient descent with random frontal slices for t-product least squares
problem \eqref{eq:tensorlinsys}.}
\label{alg:tsolve_random_slices} \begin{algorithmic}

\State \textbf{Input:} data matrix $\mathcal{A}$, measurements $\mathcal{B}$,
learning rate $\alpha$

\State \textbf{Initialize} $\mathcal{R}(0)=\mathcal{B}$, $\mathcal{X}(0)=\mathbf{0}$, $k = 0 \in \mathbb{R}^n$,
$t=0$

\While{stopping criteria not satisfied}

\State $i=\text{random}\{1,2,\cdots,n\}$ \Comment{e.g., 
uniform or leverage score sampling.}

\State $\mathcal{R}(t+1)=\mathcal{R}(t)-\tilde{\mathcal{A}}_{i}*\mathcal{X}(t)+\tilde{\mathcal{A}}_{i}*\mathcal{X}(k_i)$

\Comment{$k_i$ is the last iteration slice $i$ was selected.}

\Comment{Residual Approximation with the convention that $\mathcal{X}(t)=0$ for $t\leq0$.}

\State $\mathcal{X}(t+1)=\mathcal{X}(t)+\alpha\tilde{\mathcal{A}}_{i}^{T}*\mathcal{R}(t+1)$\Comment{Solution
approximation with 1 $i$-th frontal slice}

\State $t=t+1$

\EndWhile \end{algorithmic} 
\end{algorithm}

\begin{figure}[t]
\centering \includegraphics[width=0.75\textwidth]{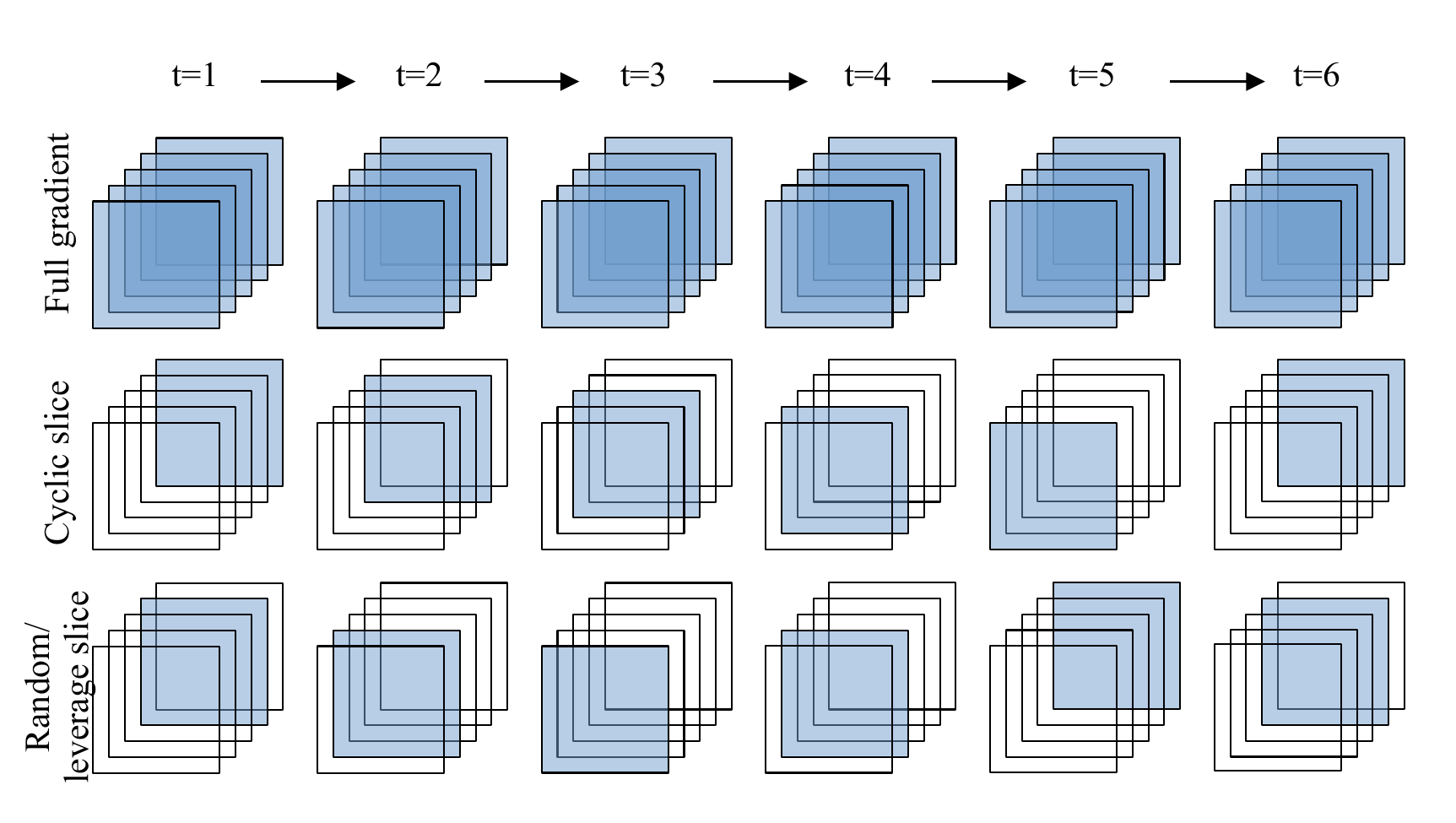}

\caption{\label{fig:gradient_schema}Algorithms \ref{alg:tsolve_full_gradient},
\ref{alg:tsolve_cyclic_slices}, \ref{alg:tsolve_random_slices} 
$t=1,2,3,4,5,6$ with a tensor $\mathcal{A}$ of frontal size $n=5$ and block size $s=1$.
We illustrate the 3-way tensor with 5 frontal slices and highlight
the slice at a given iteration in blue. %
}
\end{figure}

\subsection{Computational Considerations}
This section summarizes storage and computational complexity for the proposed methods. To highlight the impact of the number of frontal slices $n$, we let $m = \max \{n_1, n_2, n_3 \}$.

\begin{prop}
\label{prop:The-computational-complexity}The computational complexity
of t-product $\mathcal{A}*\mathcal{X}$ %
for $\mathcal{A}\in\mathbb{R}^{n_{1}\times n_{2}\times n}$ and $\mathcal{X}\in\mathbb{R}^{n_{2}\times n_{3}\times n}$. 
is $\mathcal{O}(n_1 n_2 n_3 n) \sim \mathcal{O}(m^3 n)$.
\end{prop}

\begin{proof}
We start by transforming the tensors $\mathcal{A}$ and $\mathcal{X}$ into the Fourier Domain. The tensors $\mathcal{A}$ and $\mathcal{X}$ have $n_1 n_2$ and $n_2 n_3$ tubes, respectively, each of which cost $n \log n$ to transform to and from the Fourier Domain using FFT. Thus, these step costs $\mathcal{O}((n_1 n_2 + n_2 n_3) n \log n) $ operations. Then, $n$ matrix-matrix products are computed between frontal slices of $\widehat{\mathcal{A}}$ and $\widehat{\mathcal{X}}$, which costs $\mathcal{O}(n n_1 n_2 n_3)$. When $\mathcal{A}$ has $s$ nonzero frontal slices, the complexity of the transformation to and from the Fourier Domain reduces from a dependence on $n$ to a dependence on $s$. However, $n$ matrix-matrix products are still required since frontal slices of $\mathcal{X}$ may be dense. Thus, the computational complexity remains to be $\mathcal{O}(n n_1 n_2 n_3)$.

\end{proof}

\begin{prop}
The computational complexity of each iteration in Algorithm \ref{alg:tsolve_full_gradient} 
for $\mathcal{A}\in\mathbb{R}^{n_{1}\times n_{2}\times n}$,
$\mathcal{X}\in\mathbb{R}^{n_{2}\times n_{3}\times n}$ and $\mathcal{B}\in\mathbb{R}^{n_{1}\times n_{3}\times n}$ %
is $\mathcal{O}(n_1 n_2 n_3 n) \sim \mathcal{O}(m^3 n)$. 
The storage complexity of each iteration is $\mathcal{O}(n_{1}n_{2}n+n_{1}n_{3}n+n_{2}n_{3}n) \sim \mathcal{O}(m^2 n)$. 
\end{prop}

\begin{proof}
Using Proposition \ref{prop:The-computational-complexity}, the $\mathcal{A}*\mathcal{X}(t)$
and $\mathcal{A}^{T}*\mathcal{R}(t+1)$ both require  $\mathcal{O}(n_1 n_2 n_3 n)$.
During each iteration, we need to store $\mathcal{O}(n_{1}n_{2}n+n_{1}n_{3}n+n_{2}n_{3}n)$
arrays for $\mathcal{A}$, $\mathcal{X}(t)$, and $\mathcal{R}(t+1)$. 
\end{proof}

\begin{prop}
The computational complexity of each iteration in Algorithm \ref{alg:tsolve_cyclic_slices} 
for $\mathcal{A}\in\mathbb{R}^{n_{1}\times n_{2}\times n}$, $\mathcal{X}\in\mathbb{R}^{n_{2}\times n_{3}\times n}$
and $\mathcal{B}\in\mathbb{R}^{n_{1}\times n_{3}\times n}$ is $\mathcal{O}(n_1 n_2 n_3 n)\sim\mathcal{O}(m^3 n)$.
The storage
complexity of each iteration is $\mathcal{O}(n_1 n_2 s + n_2 n_3 n^2/s + n_1 n_3 n) \sim \mathcal{O}(m^2 n (n/s))$.
\end{prop}
\begin{proof}
    In Algorithm~\ref{alg:tsolve_cyclic_slices}, only $s$ frontal slices of $\mathcal{A}$ are utilized at a time, thus, the memory complexity for  $\tilde{\mathcal{A}}_i$ is $\mathcal{O}(n_1n_2s)$ instead of $\mathcal{O}(n_1n_2n)$. Furthermore, $n/s$ historical approximations of $\mathcal{X}$ need to be stored. Thus, the total memory complexity is $\mathcal{O}(n_1 n_2 s + n_2 n_3 n^2/s + n_1 n_3 n)$.
\end{proof}

\begin{cor}
 The computational complexity of each iteration in Algorithm \ref{alg:tsolve_cyclic_slices}
($s=1$) and \ref{alg:tsolve_random_slices}  
for $\mathcal{A}\in\mathbb{R}^{n_{1}\times n_{2}\times n}$, $\mathcal{X}\in\mathbb{R}^{n_{2}\times n_{3}\times n}$
and $\mathcal{B}\in\mathbb{R}^{n_{1}\times n_{3}\times n}$ is $\mathcal{O}(n_1 n_2 n_3 n) \sim\mathcal{O}(m^3 n)$, 
The storage
complexity of each iteration is $\mathcal{O}(n_1 n_2 + n_2 n_3 n^2 + n_1 n_3 n)\sim \mathcal{O}(m^2 n^2)$. 
\end{cor}

Since $\tilde{\mathcal{A}}_i^{s}$ may be dense, the use of frontal slices does not necessarily reduce computational complexity, but reduce the communication cost.
\begin{cor} Suppose $\mathcal{A}$ has $k \ll n$ nonzero frontal slices then,
 the computational complexity of each iteration in Algorithm \ref{alg:tsolve_cyclic_slices}
(with $s=1$) and \ref{alg:tsolve_random_slices} is $\mathcal{O}(n_1 n_2 n_3 n) \sim\mathcal{O}(m^3 n)$,
for $\mathcal{A}\in\mathbb{R}^{n_{1}\times n_{2}\times n}$, $\mathcal{X}\in\mathbb{R}^{n_{2}\times n_{3}\times n}$
and $\mathcal{B}\in\mathbb{R}^{n_{1}\times n_{3}\times n}$.
The storage
complexity of each iteration is $\mathcal{O}(n_1 n_2 + n_2 n_3 k^2 + n_1 n_3 n) \sim \mathcal{O}(m^2 k^2)$. 
\end{cor}

\section{\label{sec:Convergence-Analysis}Convergence Analysis}

 Theorem~\ref{thm:conv-cyclic} provides the main convergence guarantees for the cyclic FSD algorithm (Algorithm~\ref{alg:tsolve_cyclic_slices}) when $s=1$. Theorem~\ref{thm:conv-cyc-n2} presents the special case in which $n=2$, and Corollary~\ref{cor:conv_block} extends our results to the block FSD case ($s > 1$). For an appropriately chosen learning rate $\alpha>0$, show that $\mathcal{X}(k)$
converges to the solution $\mathcal{X}_{*}$ and obtain a convergence
result in the following theorem. This is accomplished by showing approximation error $\mathcal{E}(t) = \|\mathcal{X}(t) - \mathcal{X}^* \|_F^2$ converges to 0, where $\mathcal{X}_{*}$ is the actual solution to the consistent
system \eqref{eq:tensorlinsys}. Since we are handling a finite dimensional system, it is obvious that $0<\mathcal{E}(0)<\infty$ by definiteion.

\begin{thm}[Convergence of cyclic-FSD]
\label{thm:conv-cyclic} Let $\mathcal{A}*\mathcal{X}=\mathcal{B}$
be a consistent tensor system with unique solution $\mathcal{X}^{*}$
where $\mathcal{A}\in\mathbb{R}^{n_{1}\times n_{2}\times n}$ and
$\mathcal{B}\in\mathbb{R}^{n_{1}\times n_{3}\times n.}$ Define $\epsilon_{t}$
to be the coefficient such that $\mathcal{E}(t+1)\leq\epsilon_{t}\mathcal{E}(0)$
and let 
\begin{align}
\kappa & =\max_{i=1,...n}\|\mathcal{I}-\alpha\tilde{\mathcal{A}}_{i}^{T}*\tilde{\mathcal{A}}_{i}\|_{op},\label{eq:kappa}\\
\mu & =\max_{\substack{i,j=1,...,n\\
i\neq j
}
}\|\tilde{\mathcal{A}}_{i}^{T}*\tilde{\mathcal{A}}_{j}\|_{op}.\label{eq:mu}
\end{align}
If the learning rate $\alpha$ is chosen such that   $ \kappa + \alpha \mu (n-1)<1,$
then $\mathcal{E}(t) \rightarrow 0$ as $t \rightarrow \infty$.%

\end{thm}

\begin{proof}
See Appendix~\ref{subsec:conv-cyclic}. 
\end{proof}
\begin{figure}[t]
\centering 

\begin{tabular}{ccc}
\toprule 
\multicolumn{3}{c}{$(n_{1},n_{2},n_{3})=(100,10,10)$}\tabularnewline
\midrule 
$n=2$ & $n=5$ & $n=10$\tabularnewline
\midrule
\midrule 
\includegraphics[width=0.3\textwidth]{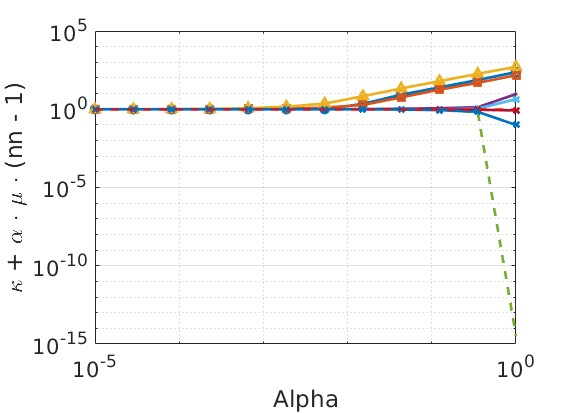} & \includegraphics[width=0.3\textwidth]{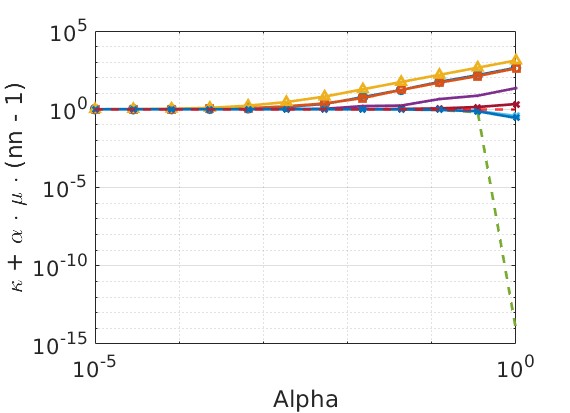} & \includegraphics[width=0.3\textwidth]{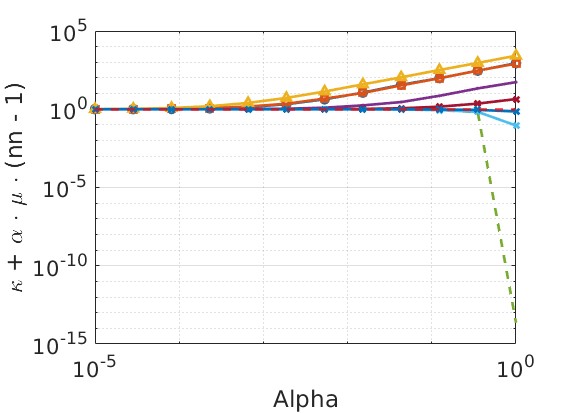} \tabularnewline
\midrule 
\multicolumn{3}{c}{\includegraphics[width=0.99\textwidth]{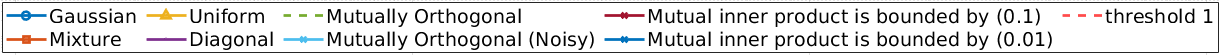} }\tabularnewline
\bottomrule
\end{tabular}

\caption{\label{fig:alpha_bound}The bound of learning rate $\alpha$ as shown in Theorem \ref{thm:conv-cyclic} for Algorithm~\ref{alg:tsolve_cyclic_slices}
to tensor $\mathcal{A}$ with (1) i.i.d Gaussian entries (2) i.i.d. Unif{[}0,1{]}
entries, and (3) an additive mixture between Gaussian and uniform entries. We also provide synthetic cases where $\mathcal{A}$ has (4) frontal slices, which are each diagonal matrices (5) mutually orthogonal frontal slices (w.r.t. matrix inner product) (6) mutual inner products between frontal slices are bounded by 0.1 (7) mutual inner products between frontal slices are bounded by 0.01  (8) mutually orthogonal frontal slices with Gaussian noise.}

\end{figure}

Since $\kappa$ is also a function of learning rate $\alpha$, we
want to ensure that assumption $\kappa + \alpha \mu (n-1) < 1$  in Theorem~\ref{thm:conv-cyclic} is feasible. When $\mu=0$, the frontal slices of the coefficient tensor $\mathcal{A}$ are mutually orthogonal frontal slices (w.r.t. matrix inner product). In this case, the sufficient condition in Theorem \ref{thm:conv-cyclic} becomes trivial when $\alpha$ is selected so that $\kappa<1$. This gives us a wide range of choices of $\alpha$ as shown by Figure \ref{fig:alpha_bound} case (5). Another extreme occurs when the frontal slices are each diagonal matrices in case (4), making $\kappa$ exceed 1 easily, and we can see that such a situation will always violate the sufficient bound. 

Beyond cases (4) and (5), cases (6) and (7) allow for approximate orthogonality. Here, the frontal slices are almost orthogonal, up to a matrix inner product threshold. As the frontal size $n$ increases, we need to set this inner product threshold to be smaller such that $\mu(n-1)$ falls in a reasonable range. Another scenario case (8) is created by perturbing the case (4) using a standard Gaussian noise tensor (like case (1)) multiplied by $\frac{1}{n^3}$, also creating a scenario where our sufficient bound for $\alpha$ can be met.

Unfortunately, the more usual cases like (1) to (3) in Figure~\ref{fig:alpha_bound} will not satisfy our sufficient condition in Theorem \ref{thm:conv-cyclic}; yet we will see in the next section (i.e., Figure \ref{fig:Performance-of-Algorithm-diffsystem}) that even if our sufficient condition on the learning rate $\alpha$ does not hold, Algorithm \ref{alg:tsolve_cyclic_slices} and its variants still yields reasonable convergence rates with fixed $\alpha$. This makes it a well-defined future work to find the necessary condition for the learning rate $\alpha$.

Corollary~\ref{cor:conv_block} presents the convergence of the blocked case $2\leq s<n$ in Algorithm~\ref{alg:tsolve_cyclic_slices}. 
\begin{cor}
\label{cor:conv_block}For $n\geq2$ and $2\leq s<n$ in Algorithm~\ref{alg:tsolve_cyclic_slices},
we define $\tilde{\mathcal{A}}_{i}^{s}\coloneqq\tilde{\mathcal{A}}_{i\cdot(s-1)+1}+\tilde{\mathcal{A}}_{i\cdot(s-1)+2}+\cdots+\tilde{\mathcal{A}}_{i\cdot s}$
to be the sum of $s$ padded slices and
\begin{align}
\kappa(s) & =\max_{i=1,...n/s}\|\mathcal{I}-\alpha\tilde{\mathcal{A}}_{i}^{sT}*\tilde{\mathcal{A}}_{i}^{s}\|_{op},\label{eq:kappa-s}\\
\mu(s) & =\max_{\substack{i,j=1,...,n/s\\
i\neq j
}
}\|\tilde{\mathcal{A}}_{i}^{sT}*\tilde{\mathcal{A}}_{j}^{s}\|_{op}.\label{eq:mu-s}
\end{align}
Suppose that the learning rate $\alpha$ is chosen such that 
$$ \kappa(s) + \alpha \mu(s) (n/s-1)<1.$$
Then, we have $\lim_{t\rightarrow\infty}\mathcal{E}(t)=0$.
\end{cor}

\begin{proof}
See Appendix \ref{subsec:Blocked-Case}. 
\end{proof}

When $n=2$, by Theorem~\ref{thm:conv-cyclic}, we have that Algorithm~\ref{alg:tsolve_cyclic_slices} converges to the solution to the consistent system \eqref{eq:tensorlinsys} as $t \rightarrow \infty$ as long as $\alpha$ is chosen such that $\kappa < 1$ and $\alpha < \frac{1-\kappa}{\mu}$. %
In the $n=2$ case, we can also place additional constraints on the learning rate $\alpha$ to determine the convergence rate per $n$ iterations. 

\begin{thm}
\label{thm:conv-cyc-n2} 
Define $\kappa_{1}=\|\mathcal{I}-\alpha\tilde{\mathcal{A}}_{1}^{T}*\tilde{\mathcal{A}}_{1}\|_{op}$,
$\kappa_{2}=\|\mathcal{I}-\alpha\tilde{\mathcal{A}}_{2}^{T}*\tilde{\mathcal{A}}_{2}\|_{op}$
and $\mu=\|\tilde{\mathcal{A}}_{1}^{T}*\tilde{\mathcal{A}}_{2}\|_{op}$. If $\alpha<\min\{\frac{1-\kappa_{2}}{\mu},\frac{1-\kappa_{1}\kappa_{2}}{\mu(1+\kappa_{2})}\}$
then we have $\lim_{t\rightarrow\infty}\mathcal{E}(t)=0$. 
Furthermore, assume that $\mathcal{M}(0)<\infty$ for 
$$\mathcal{M}(t) = \max\{\mathcal{E}(2t), \mathcal{E}(2t+1)\},$$ $\mathcal{M}(t)$ converges at a rate of $C = \max\{\kappa_{2}+\alpha\mu,\kappa_{1}\kappa_{2}+(1+\kappa_{1})\alpha\mu\} <1$.
\end{thm}

\begin{proof}
See Appendix \ref{sec:cyc-n2}. 
\end{proof}
The proof of Theorem~\ref{thm:conv-cyc-n2} is provided in Appendix
\ref{sec:cyc-n2}. Note that the $\kappa_{1},\kappa_{2}$ contains
the learning rate $\alpha$ and therefore the sufficient condition
$\alpha<\min\{\frac{1-\kappa_{2}}{\mu},\frac{1-\kappa_{1}\kappa_{2}}{\mu(1+\kappa_{2})}\}$
requires solving this inequality. It implies that the range of $\alpha$ where an algorithm converges
is determined by the $\|\tilde{\mathcal{A}}_{i}^{T}*\tilde{\mathcal{A}}_{i}\|$
($i=1,2$) and $\mu=\|\tilde{\mathcal{A}}_{1}^{T}*\tilde{\mathcal{A}}_{2}\|_{op}$,
which are the norms of each (padded) frontal slice and correlations
between the (padded) frontal slices of the coefficient tensor in the
system \eqref{eq:tensorlinsys}.

In Theorem~\ref{thm:conv-cyc-n2}, in addition to the limiting error converging to zero,
we can also observe that the norms $\|\tilde{\mathcal{A}}_{i}^{T} * \tilde{\mathcal{A}}_{i}\|$
($i=1,2$) and correlation $\mu$ between frontal slices jointly determine
the convergence rate of the Algorithm~\ref{alg:tsolve_cyclic_slices}.
When the norm is larger, the $\kappa$ tends to be smaller, and the
convergence rate becomes smaller; when the correlation $\mu$ is smaller,
the convergence rate becomes smaller.
The same idea also applies when we compare the convergence rates
of $s=1,s>1$ cases. The convergence rates are dominated by $\kappa(s)^{t}<\kappa^{t}$,
indicating that single-slice descent indeed sacrifices the convergence
rate to only work with one frontal slice of $\mathcal{A}$ at a time.

When the system is inconsistent, there will be an additional error
term in the bound, which creates a convergence horizon. In this more
realistic scenario where the system is corrupted by noise, the choice
of $\mathcal{X}(0)$ and hence $\mathcal{E}(0)$ will affect the convergence
behavior. %

\begin{thm}
\label{thm:conv-cyclic-noise} %
Let $\kappa,\mu$ be defined as in Theorem \ref{thm:conv-cyclic} and suppose that Algorithm~\ref{alg:tsolve_cyclic_slices} with $s=1$ is applied to the inconsistent linear system $\mathcal{A} * \mathcal{X} = \mathcal{B} + \mathcal{B}_e,$
with solution $\mathcal{X}^*$ such that $\mathcal{A} * \mathcal{X}^* = \mathcal{B}$. Define $\eta_t \in \mathbb{R}$ such that $\mathcal{E}(t+1) \leq \eta_t \mathcal{E}(t) + \alpha \eta_e$ 
where $\eta_e = \max_{i=1,\cdots,n}\left\Vert \tilde{\mathcal{A}}_{i}\right\Vert _{op}\left\Vert \mathcal{B}_{e}\right\Vert$. If the learning rate $\alpha$ is selected such that $\kappa+\alpha\mu(n-1)+\frac{\alpha}{\mathcal{E}(0)}\eta_{e}  <1$ 
then as $t \rightarrow \infty$, $\mathcal{E}(t) \rightarrow \alpha \eta_e$.
\end{thm}

\begin{proof}
See Appendix \ref{subsec:conv-cyclic-noise}. 
\end{proof}

When frontal slices are selected at random, without replacement, the same proof argument follows by relabeling the frontal slices after each $n$ iterations. However, the proof of the randomized case with replacement is nontrivial without additional or restrictive assumptions. We leave the proofs of this case as future work.

\section{\label{sec:Experiments}Experiments}

In this section, we perform a collection of numerical experiments
on synthetic and real-world data sets to demonstrate and compare the
application of the proposed methods. %
In the our experiments, unless otherwise noted, we fix $n_{1}$,
$n_{2}$, $n_{3}$, and $n$. The first set of experiments pertain
only to the performance of FSD. The first is a study of the
algorithm with varying step sizes, the second investigates the algorithm's
performance when key parameters in our theoretical guarantees ($\kappa$
and $\mu$) vary, and the last investigates the performance deterioration
as $n$ grows. The second two experiments compare the variations of
FSD and then FSD with other sketching approaches. The last two experiments
evaluate the performance of FSD on image and video deblurring, where
only a small number of frontal slices of $\mathcal{A}$ are non-zero
due to the nature of the deblurring operator, thus making it a prime
case where FSD can be efficiently applicable. 

\subsection{Synthetic Experiments}

For all synthetic experiments, unless otherwise noted, the entries
of $\mathcal{A}$ and $\mathcal{X}$ are drawn i.i.d. from a standard
Gaussian distribution, $\mathcal{X}$ is normalized to have unit Frobenius
norm, and we set $\mathcal{B}=\mathcal{A}*\mathcal{X}$ to generate
consistent systems.

Our first synthetic experiment demonstrates the performance of Algorithm~\ref{alg:tsolve_cyclic_slices}
on a Gaussian system with varying learning rates and dimensions. The
three different dimension configurations provide insight into how
the tensor's structure affects the methods' performance. In general,
we observe that when the frontal size ($n$) is moderate ($n=10$), the
performance varies slightly among methods but follows a general trend
of decreasing error with more iterations. When the frontal size ($n$)
is increased ($n=20$), the performance differences become more pronounced.
When the column size ($n_{2}$) is larger ($n_{2}=20$),
leverage score sampling 
tends to perform better,
suggesting that this method handles higher column dimensions effectively.

For $\alpha$=0.5 (except for TRK), all FSD methods
show a gradual decrease in error, but the convergence is slower compared
to higher learning rates. 
  For
$\alpha=1$, the convergence speeds up, and the TRK method still 
consistently outperforms others, achieving the lowest approximation
error. For $\alpha=2$, the convergence becomes fast, but there are instances of oscillation, which
suggests instability at higher learning rates. The cyclic FSD demonstrates the best performance, maintaining stability and
achieving low errors. When we keep increasing the learning rate of FSD methods to $\alpha=4$, then cyclic FSD along with its variants outperforms TRK in most cases. %

Overall, the results indicate that the cyclic and leverage
score sampling methods perform well, particularly when the column
size is large. The TRK method shows competitive performance at lower
learning rates but becomes unstable at higher learning rate $\alpha$.
Random sampling and cyclic methods perform adequately but generally
lag behind leverage score sampling in terms of convergence
speed and final error. These findings suggest that choosing the appropriate
method and tuning the learning rate based on the tensor's dimensions
can significantly improve performance in solving large-scale tensor
systems.
\newpage
\begin{figure}[H]
\centering %

\begin{tabular}{ccc}
\toprule 
&$(n_{1},n_{2},n_{3},n)$&
\tabularnewline
$(100,20,10,10)$  & $(100,10,20,10)$  & $(100,10,10,20)$\tabularnewline
\midrule
\midrule 
 & $\alpha=0.5$  & \tabularnewline
\midrule 
\includegraphics[width=0.3\textwidth]{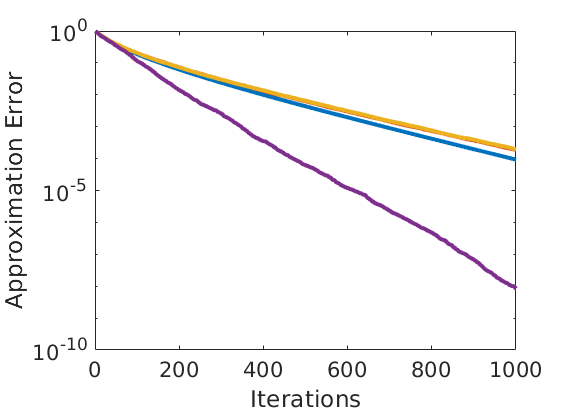}  & \includegraphics[width=0.3\textwidth]{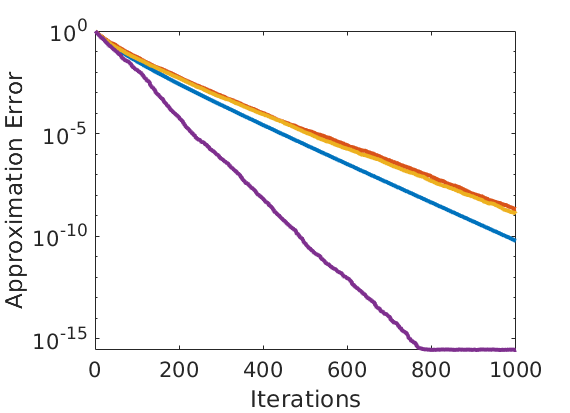}  & \includegraphics[width=0.3\textwidth]{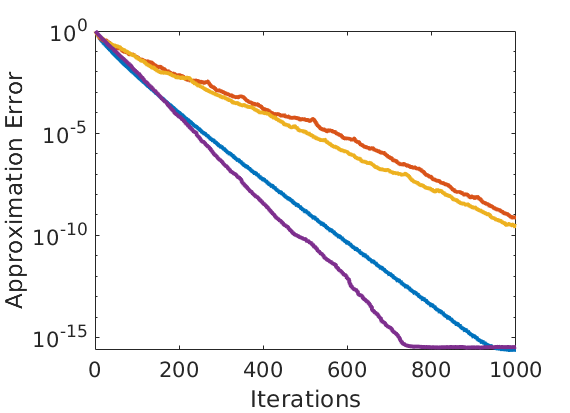}\tabularnewline
\midrule 
 & $\alpha=1$  & \tabularnewline
\midrule 
\includegraphics[width=0.3\textwidth]{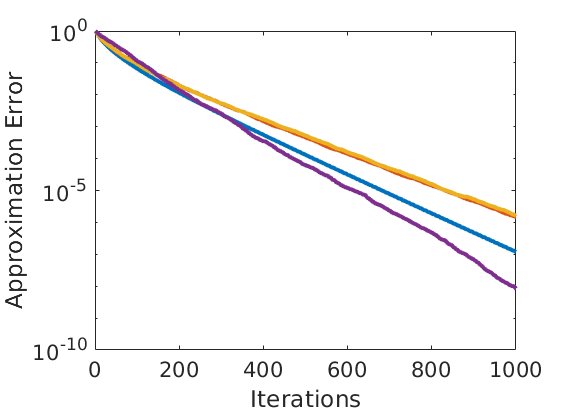}  & \includegraphics[width=0.3\textwidth]{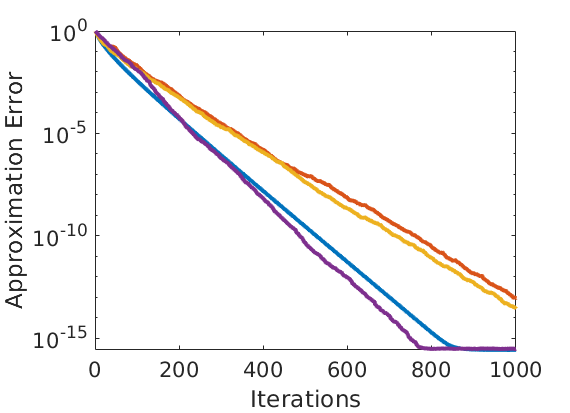}  & \includegraphics[width=0.3\textwidth]{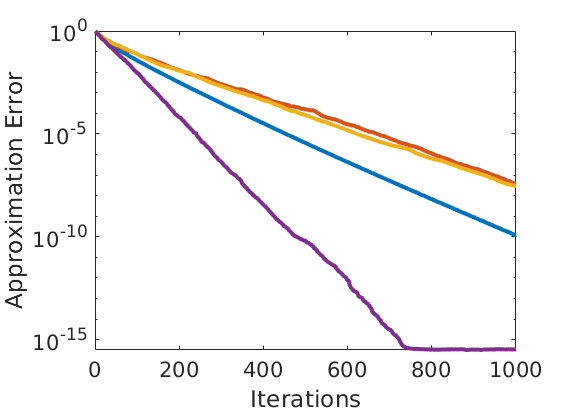}\tabularnewline
\midrule 
 & $\alpha=2$  & \tabularnewline
\midrule 
\includegraphics[width=0.3\textwidth]{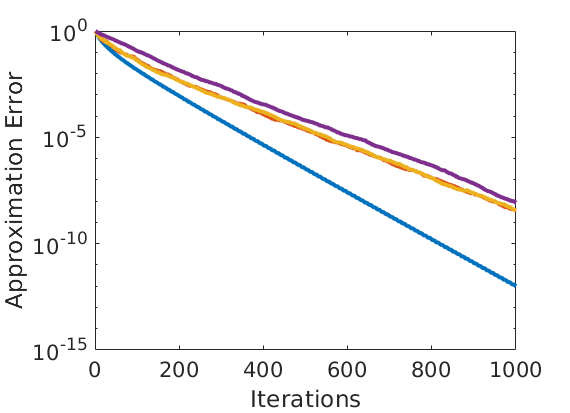}  & \includegraphics[width=0.3\textwidth]{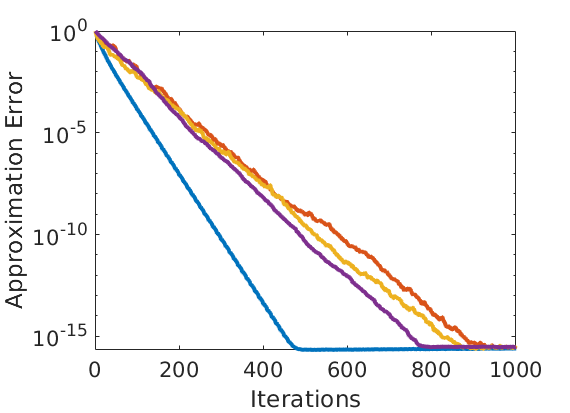}  & \includegraphics[width=0.3\textwidth]{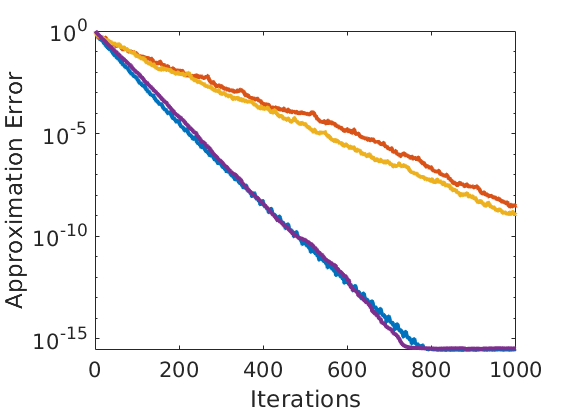}\tabularnewline
\midrule 
 & $\alpha=4$  & \tabularnewline
\midrule 
\includegraphics[width=0.3\textwidth]{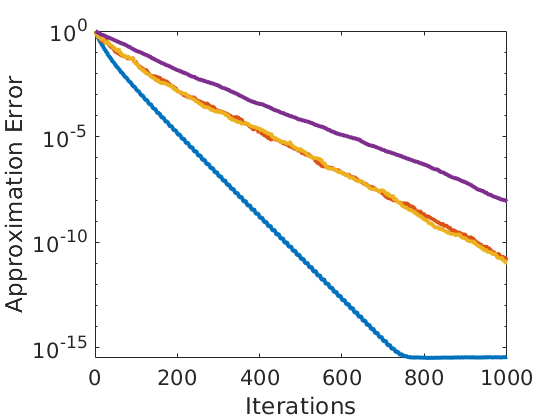}  & \includegraphics[width=0.3\textwidth]{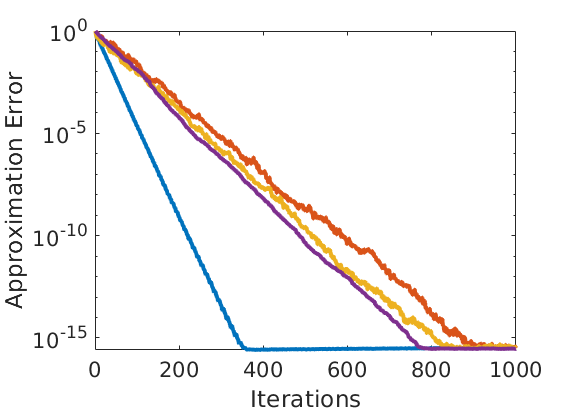}  & \includegraphics[width=0.3\textwidth]{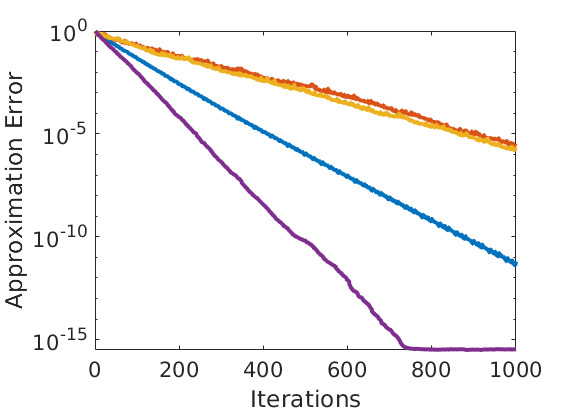}\tabularnewline
\bottomrule
\end{tabular}\\
 \includegraphics[width=0.5\textwidth]{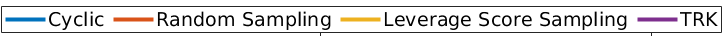}\caption{\label{fig:Generate--and}\label{fig:exp_variations} Generate $\mathcal{A},\mathcal{X}$
and compute consistent $\mathcal{B}=\mathcal{A}*\mathcal{X}$ for
the synthetic system for varying dimension sizes. In each panel, we
show the final approximation error $\mathcal{E}(T)$ using different
methods (Frontal: Algorithm \ref{alg:tsolve_cyclic_slices}; Sample
(Random): Algorithm \ref{alg:tsolve_random_slices} with uniform random
sampling of slices; Sample (Leverage): Algorithm \ref{alg:tsolve_random_slices}
with leverage score sampling of slices; 
TRK: \cite{ma2022randomized}) and learning rate $\alpha=0.001$ and
the maximal number of iterations of 1000 steps. Comparison of performance
of different iterative methods for solving tensor systems when $n_{1}=100$,
$n_{2}=20$, $n_{3}=10$, $n=10$, and $\alpha=\{0.5,1,2,4$\} for frontal descent methods %
}%
\label{fig:comparisonAll}
\end{figure}

Figure~\ref{fig:varyn} shows the performance of Algorithm~\ref{alg:tsolve_cyclic_slices}
when varying the number of frontal slices. In this experiment, we
plot the final approximation error after 1000 iterations and averaged
over 20 runs for each values of $n$ and present the average approximation
error $\|\mathcal{X}(1000)-\mathcal{X}_{*}\|_{F}$. We expect that
as the number of frontal slices increases, the per-iteration approximation
quality of the residual deteriorates. Thus, it is expected that the
final approximation error after 1000 iterations and $\alpha=0.001$
increases as $n$ increases, as shown in Figure~\ref{fig:varyn}.
\begin{figure}[ht!]
\centering \includegraphics[width=0.45\textwidth]{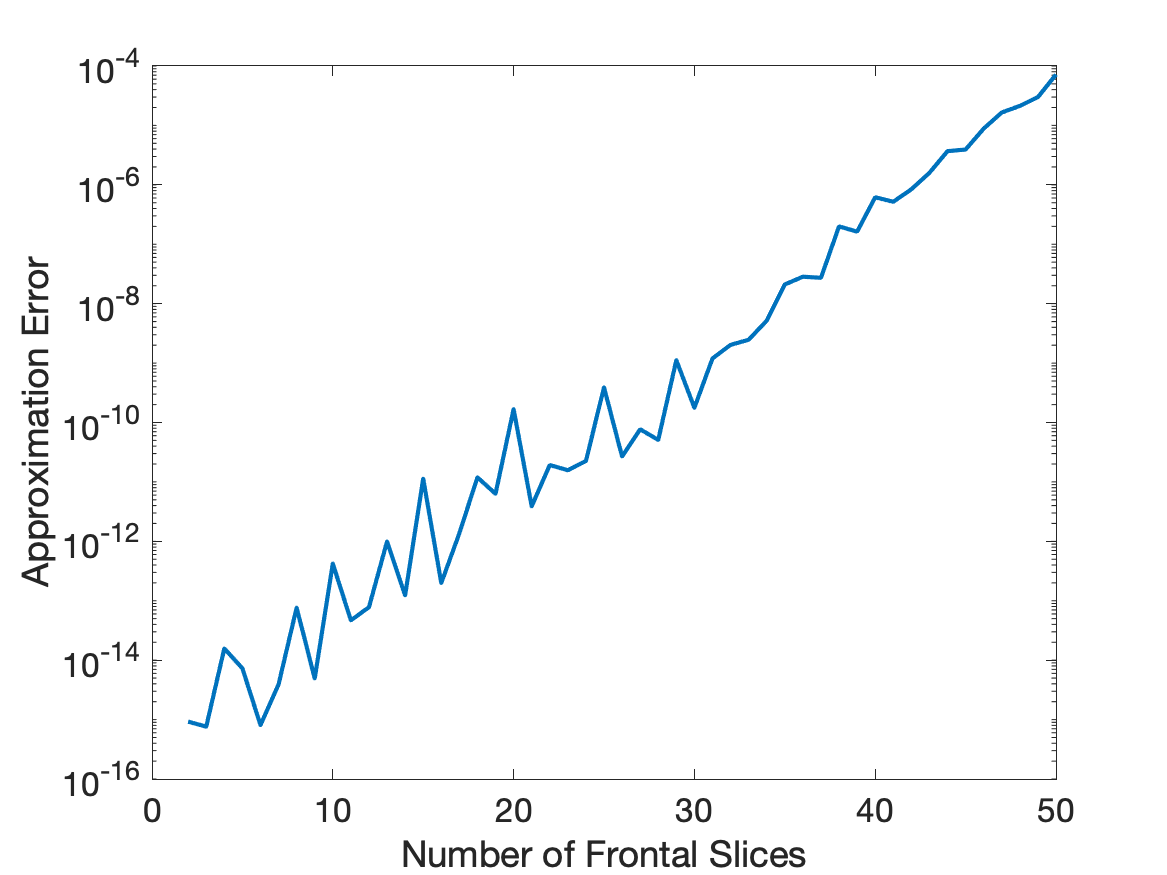}
\includegraphics[width=0.45\textwidth]{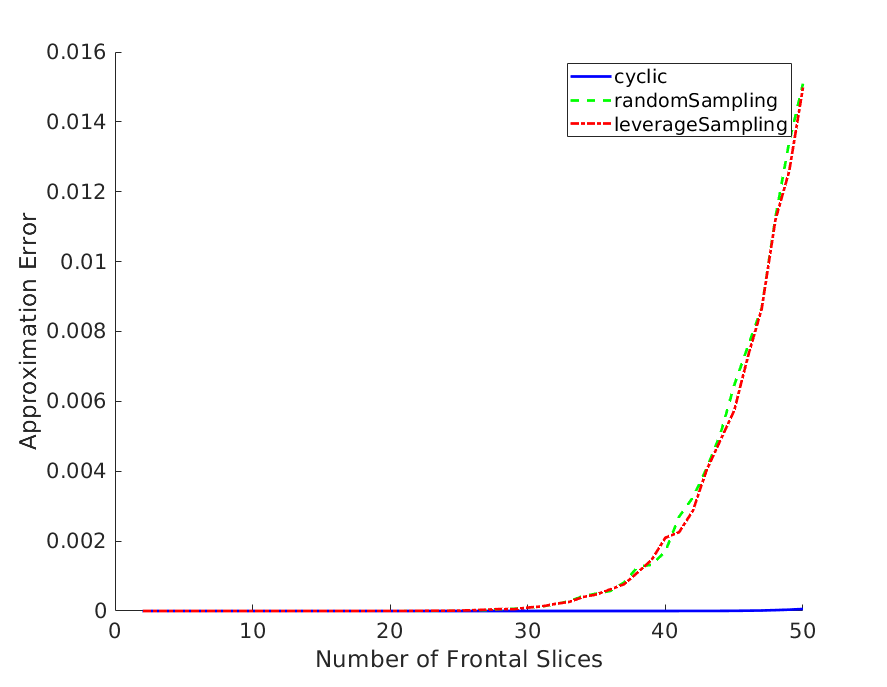}
\caption{Performance of Algorithm~\ref{alg:tsolve_cyclic_slices} (left),
\ref{alg:tsolve_random_slices} (right: joint comparison between cyclic,
uniform and leverage score sampling) on an i.i.d. Gaussian system
with $n_{1}=100$, $n_{2}=20$, $n_{3}=10$, $\alpha=0.001$, and
varying $n$ averaged over 20 different seeds. As the 
 frontal size increases, the approximation errors at the end of (the same maximal number of) iterations also increase.}
\label{fig:varyn} 
\end{figure}

Our next experiment considers different types of measurement tensors
$\mathcal{A}$, which exhibit different $\kappa$ and $\mu$. In this
experiment, we fix $n_{1}=100$, $n_{2}=20$, $n_{3}=10$, $n=10$,
and $\alpha=0.0001$ while applying Algorithm~\ref{alg:tsolve_cyclic_slices}
to tensors with (1) i.i.d Gaussian entries (2) i.i.d. Unif{[}0,1{]}
entries and (3) an additive mixture between Gaussian and uniform entries.
The accompanying table in Figure~\ref{fig:varysys} shows the values
of $\kappa$ and $\mu$ for the randomly generated systems while Figure~\ref{fig:varysys}
presents the performance of Algorithm~\ref{alg:tsolve_cyclic_slices}
for the different types of systems. Note that for Gaussian systems,
where $\mu$ is small, the proposed method works quite well, whereas
it converges slower for systems in which $\mu$ is larger, such as
the Uniform system. In particular, we can compute the numerical approximation
(via an optimizer for obtaining the maximum over the sphere $\|\mathcal{X}\|_{op}$)
of the quantities \eqref{eq:kappa} and \eqref{eq:mu}. The scatter
plot in Figure \ref{fig:Performance-of-Algorithm-diffsystem} visualizes
systems with different $\kappa$ and $\mu$ values, generated using
Gaussian, Mixture, and Uniform distributions. The x-axis represents
$\kappa$ and the y-axis represents $\mu$, with dot colors indicating
the final approximation error. Gaussian systems, represented by circles,
cluster towards lower $\kappa$ and $\mu$ values, displaying relatively
low approximation errors (darker blue colors). Mixture systems, represented
by squares, have slightly higher $\kappa$ and $\mu$ values than
Gaussian systems but still maintain low errors. Uniform systems, shown
as triangles, spread out with significantly higher $\kappa$ and $\mu$
values, exhibiting a wider range of approximation errors, with some
reaching higher values (yellow shades). This plot effectively demonstrates
that Gaussian and Mixture systems yield lower $\kappa$ and $\mu$
values and lower approximation errors, while Uniform systems display
higher values and more variability in their errors, highlighting the
impact of system type on tensor approximation performance. Note that
the learning rate $\alpha$ does not fall in the sufficient convergence
horizon in Theorem \ref{thm:conv-cyclic}, yet we still observe convergence,
which indicates that a sharper horizon is possible. 
\begin{figure}[h!]
\centering

\includegraphics[width=0.45\textwidth]{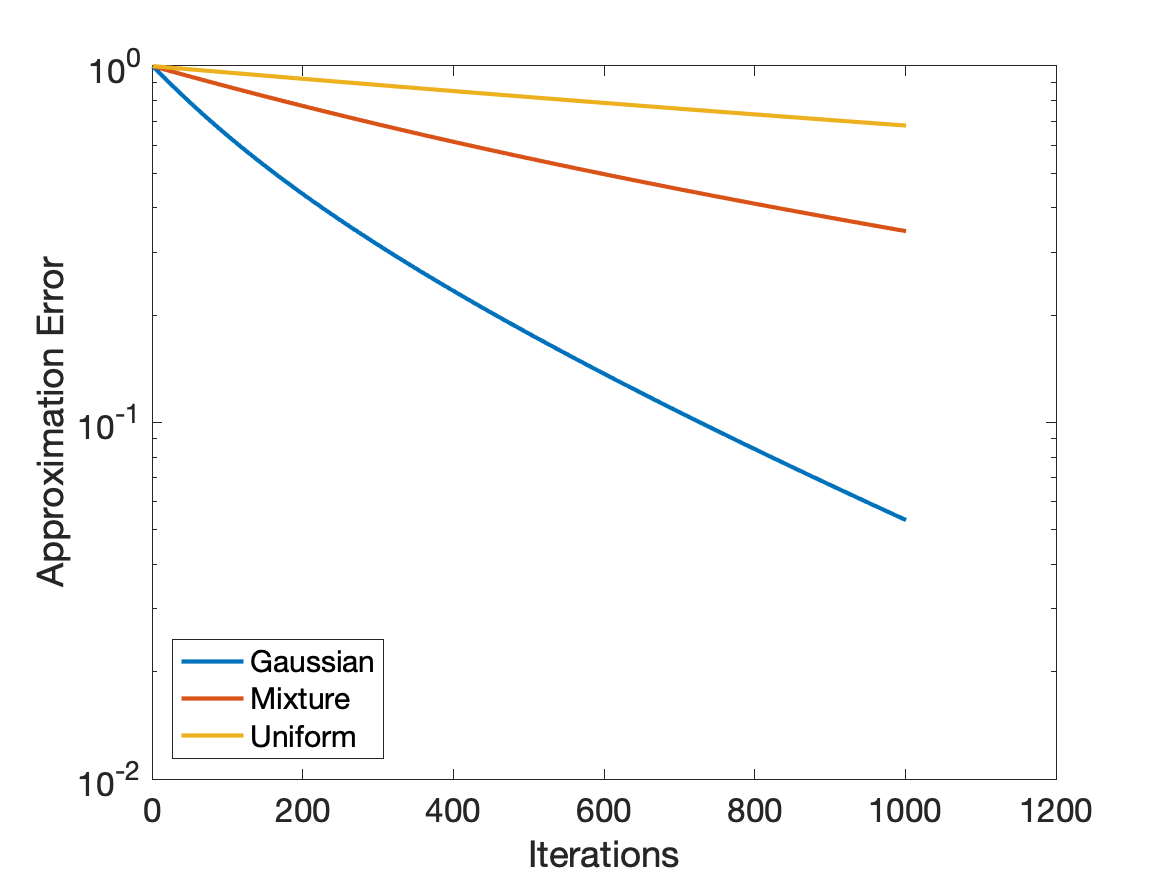}\includegraphics[width=0.45\textwidth]{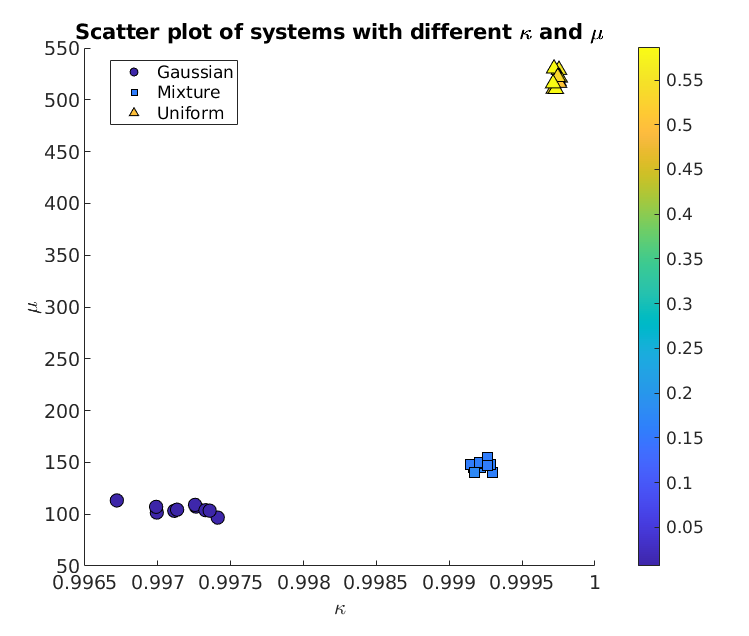}

\caption{\label{fig:Performance-of-Algorithm-diffsystem}Performance of Algorithm~\ref{alg:tsolve_cyclic_slices}
(block size=1) with different types of tensors $\mathcal{A}$ when
$n_{1}=100$, $n_{2}=20$, $n_{3}=10$, $n=10$, and $\alpha=5\times10^{-5}$.
Left: We provide the numerical approximation to the operator norms
$\kappa$ and $\mu$ for each system as defined in \ref{eq:kappa}
and \ref{eq:mu}. Right: The scatter plot visualizes the systems with
different $\kappa$ and $\mu$ values, generated using Gaussian, Mixture,
and Uniform distributions. Each dot's color indicates the final approximation
error, with the color bar on the right showing the error scale. }
\label{fig:varysys} 
\end{figure}

Figure~\ref{fig:exp_varyblock} presents the performance of Algorithm~\ref{alg:tsolve_cyclic_slices}
with varying block sizes $s$. The measurement tensor $\mathcal{A}$
is i.i.d. Gaussian with $n=10$ frontal slices. Thus, when $s=10$,
Algorithm~\ref{alg:tsolve_cyclic_slices} with blocking is equivalent
to Algorithm~\ref{alg:tsolve_full_gradient}. Intuitively, we
have more information when we have larger block sizes, larger block sizes allow the algorithm to converge faster since the residual approximations
improve as $s$ increases. In terms of convergence horizon, we need
fewer assumptions to check for $\kappa,\mu$ and hence enjoy a wider
horizon of convergence (See details in Appendix \ref{subsec:Blocked-Case}).
\begin{figure}[h]
\centering \includegraphics[width=0.3\textwidth]{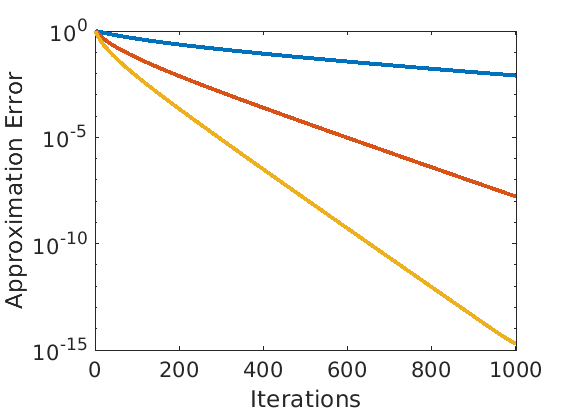}
\includegraphics[width=0.3\textwidth]{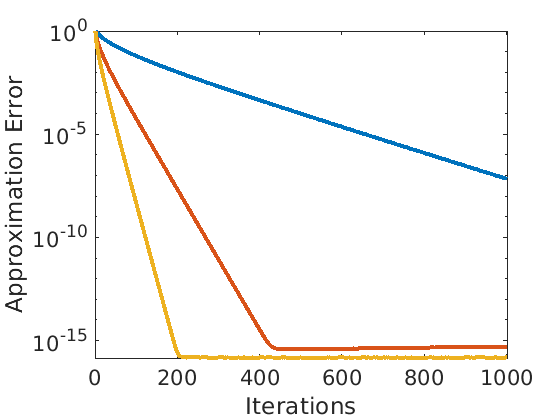}
\includegraphics[width=0.3\textwidth]{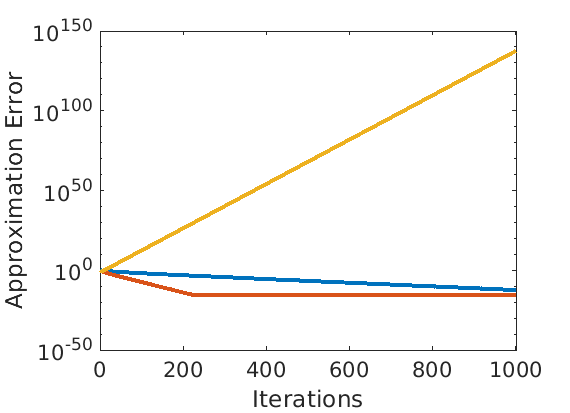}
\\
\includegraphics[width=0.2\textwidth]{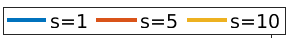}\caption{Performance of Algorithm~\ref{alg:tsolve_cyclic_slices} with different
frontal block sizes $s$ when $n_{1}=100$, $n_{2}=20$, $n_{3}=10$,
$n=10$, and $\alpha=0.0001,0.0005,0.001$.}
\label{fig:exp_varyblock} 
\end{figure}

Our last synthetic experiment compares the variations of the proposed
method with the TRK algorithm~\cite{ma2022randomized}. Since the
TRK algorithm assumes row slices are normalized, we incorporate the
row normalization for each row slice and increase the learning rate
$\alpha=3$ for all variations of frontal slice descent. Figure~\ref{fig:exp_variations}
presents the performance of Algorithm~\ref{alg:tsolve_cyclic_slices}
(``Cyclic"), Algorithm~\ref{alg:tsolve_random_slices} with leverage
score sampling, Algorithm~\ref{alg:tsolve_random_slices} with simple
random sampling, %
and TRK on a Gaussian system when $n_{1}=100$,
$n_{2}=20$, $n_{3}=10$, and $n=10$. In this case, we observe that
Algorithm~\ref{alg:tsolve_cyclic_slices} and Algorithm~\ref{alg:tsolve_random_slices}
perform similarly
as shown in Figure \ref{fig:Generate--and},
and most importantly, that all three methods are converging to the
solution. 
It should be noted that the choice of learning rate impacts
which method outperforms the other (choosing a smaller learning rate
may lead to Algorithm~\ref{alg:tsolve_cyclic_slices} to converge
slower than TRK, see, for example, Figure~\ref{fig:experiments_HST}).

\begin{table}
\centering

\begin{tabular}{ccccc}
\toprule 
$n$  & 1  & 10  & 100  & 1000 \tabularnewline
\midrule
\midrule 
Frontal: Algorithm \ref{alg:tsolve_cyclic_slices}  & 0.016734  & 0.030362  & 0.089839  & 0.743340 \tabularnewline
\midrule 
Sample (Random): Algorithm \ref{alg:tsolve_random_slices}  & 0.015908  & 0.031074  & 0.096795  & 0.739575 \tabularnewline
\midrule 
Sample (Leverage): Algorithm \ref{alg:tsolve_random_slices}  & 0.015652  & 0.030578  & 0.092100  & 0.730805 \tabularnewline
\midrule 
TRK: \cite{ma2022randomized}  & 0.003110  & 0.003010  & 0.015947  & 0.013112 \tabularnewline
\bottomrule
\end{tabular}

\caption{\label{tab:The-system-time}The system wall-clock time (seconds via
tic-toc) of each algorithms for 1000 iterations on a system with $n_{1}=n_{2}=n_{3}=100,$
but with increasing $n$. This time is averaged over 5 repeated experiments. }
\end{table}

\begin{rem}
The TRK algorithm uses exact residual information for its updates,
which can lead to more precise and potentially faster convergence
as shown in Table \ref{tab:The-system-time}. In contrast, the method
discussed in this paper does not utilize exact residual information,
which may lead to slower convergence or less accurate updates. The
frontal sketching methods discussed in Algorithms \ref{alg:tsolve_cyclic_slices} and
\ref{alg:tsolve_random_slices}  %
may perform worse than TRK because they doesn't use exact residual
information (See Figure \ref{fig:comparisonAll}). 
\end{rem}

\begin{rem}
The varying $n_{3}$ should not affect the convergence result that
much, since our main theorems (in Section \ref{sec:Algorithmic-Approaches})
bounds the convergence rate using information of the tensor $\mathcal{A}$
only, which only involves $n_{1},n_{2},n$. The performance of matrix
sketching methods depends on the aspect ratio, specifically the number
of rows ($n_{1}$) to columns ($n_{2}$). For tensor sketching, similar
principles apply, involving the row/frontal ($n_{1}/n$) and column/frontal
($n_{2}/n$) ratios. These ratios affect both computational efficiency
and convergence rates. In tensor algorithms like frontal slice descent,
higher row/frontal or column/frontal ratios can lead to better-conditioned
sub-problems and faster convergence, but at a higher computational
cost per iteration. Optimizing these aspect ratios is crucial for
improving the performance of tensor-based sketching methods in practical
applications. For example, if $n_{1}=100$ (rows) and $n_{2}=20$
(columns), and given the effective frontal size of $n=10$, the effective
row/frontal ratio is $100/10=10$, and the effective column/frontal
ratio is $20/10=2$. Since the effective column/frontal ratio is smaller,
column sketching might be more appropriate despite the large nominal
frontal size. This approach considers both actual and effective sizes,
optimizing the sketching method based on correlation.

\end{rem}

\subsection{Real World Application: Deblurring}

In many imaging problems, particularly those involving inverse operations
such as image deblurring or reconstruction, the problem is often ill-posed.
This typically happens because the forward model (the process that
generates the observed data from the true image) is not invertible
or is sensitive to noise. For example, if multiple images can produce
similar observed data, then the problem lacks a unique solution. As
another example, if small changes or noise in the observed data lead
to large changes in the reconstructed image, then the problem lacks
stability. An example of such an imaging problem is \emph{deblurring},
where the objective is to recover the original version of an image
that has been blurred. Traditional methods often struggle with recovering
high-quality images from blurred observations, especially when dealing
with tensor data structures.

The t-product's convolutional nature is particularly useful in imaging
problems because it aligns with the structure of typical imaging operations.
Convolution operations in imaging can be represented as circulant
matrices, which naturally align with the t-product formulation. This
alignment is particularly beneficial for ill-posed imaging problems,
where the convolution represents the blurring or transformation applied
to an image. A blurring operator, designed using Gaussian kernels
(of fixed band-width and noise variance on the normalized scale),
is applied to each slice of the resized tensor to simulate a blurred
video. The algorithms aim to solve the tensor equation $\mathcal{A}*\mathcal{X}=\mathcal{B}$,
where $\mathcal{B}$ is the observed blurred tensor, $\mathcal{A}$
is the blurring operator (which can be represented as a tensor operation
via t-product), and $\mathcal{X}$ is the original noiseless tensor
in (1) or (2). Error metrics are calculated for each slice by comparing
the recovered slices with the original slices and then averaged over
all slices. The error metrics for each algorithm are averaged over
all slices to provide a comprehensive assessment. Lower MSE values
and higher PSNR and SSIM values indicate better recovery quality.
The metrics are defined as follows: 
\begin{itemize}
\item \textbf{MSE} measures the average squared difference between the original
and recovered pixel values. Lower values indicate closer resemblance
to the original image. 
\item \textbf{PSNR} is a logarithmic measure of the ratio between the maximum
possible pixel value and the power of the noise (error). Higher values
indicate less noise in the recovered image. 
\item \textbf{SSIM} evaluates the structural similarity between the original
and recovered images, taking into account luminance, contrast, and
structure. Values range from -1 to 1, with higher values indicating
better structural similarity. 
\end{itemize}
Given the t-product structure, \cite{reichel2022tensor} develops
several algorithms for solving such ill-posed t-product linear systems,
including Tikhonov Regularization with Arnoldi (tAT), Generalized
Minimal Residual (tGMRES) Methods and their global variants. However,
these are regularized methods; we will show how the performance of
exact sketching solvers behave in this experiment. In this section,
we compare the use of frontal slices and row slices to image deblurring
applications. The generation of the blurring operator $\mathcal{A}$
closely follows that outlined in ~\cite{reichel2022tensor}, and
we present the details here for completion. Given parameters $\sigma$,$N$,and
$B$, the tensor $\mathcal{A}\in\mathbb{R}^{N\times N\times N}$ is
generated with frontal slices 
\[
\mathcal{A}_{i}=A_{1}(i,1)A_{2}\quad\quad i=1,...,N,
\]
where matrices $A_{1}$ and $A_{2}$ are generated by the Matlab commands:
\begin{align*}
\mathtt{z}_{1} & =\mathtt{[exp(-([0:band-1].^{2})/(2\sigma^{2})),zeros(1,N-B)]},\\
{\tt z_{2}} & ={\tt [z_{1}(1)~fliplr(z_{1}(end-length(z_{1})+2:end))]},\\
A_{1} & =\frac{1}{\sigma\sqrt{2\pi}}{\tt toeplitz(z_{1},z_{2})},\\
A_{2} & =\frac{1}{\sigma\sqrt{2\pi}}\mathtt{toeplitz(z_{1})}
\end{align*}
and $B$ and $\sigma$ determine the band size and variation in the
Gaussian blur, respectively. It should be noted that the frontal slices
are extremely ill-conditioned: the first $B+1$ frontal slices have
conditioning $\approx10^{6}$, and the remaining frontal slices are
all zeros. Since the proposed methods are frontal slice-based, we
only need to cycle through the first $B+1$ frontal slices of $\mathcal{A}$.

Our experiments demonstrate that tensor recovery algorithms can effectively
deblur and recover images from blurred observations. The metrics used
provide a quantitative assessment of the recovery quality, with lower
MSE and higher PSNR and SSIM values indicate better performance. These
findings are valuable for applications in image processing and computer
vision, where high-quality image recovery is essential.

\begin{figure}[h]
\centering %

\textbf{100 by 100 resized HST Satellite Image from \cite{reichel2022tensor}}\\

\includegraphics[width=0.6\textwidth,height=3.5cm]{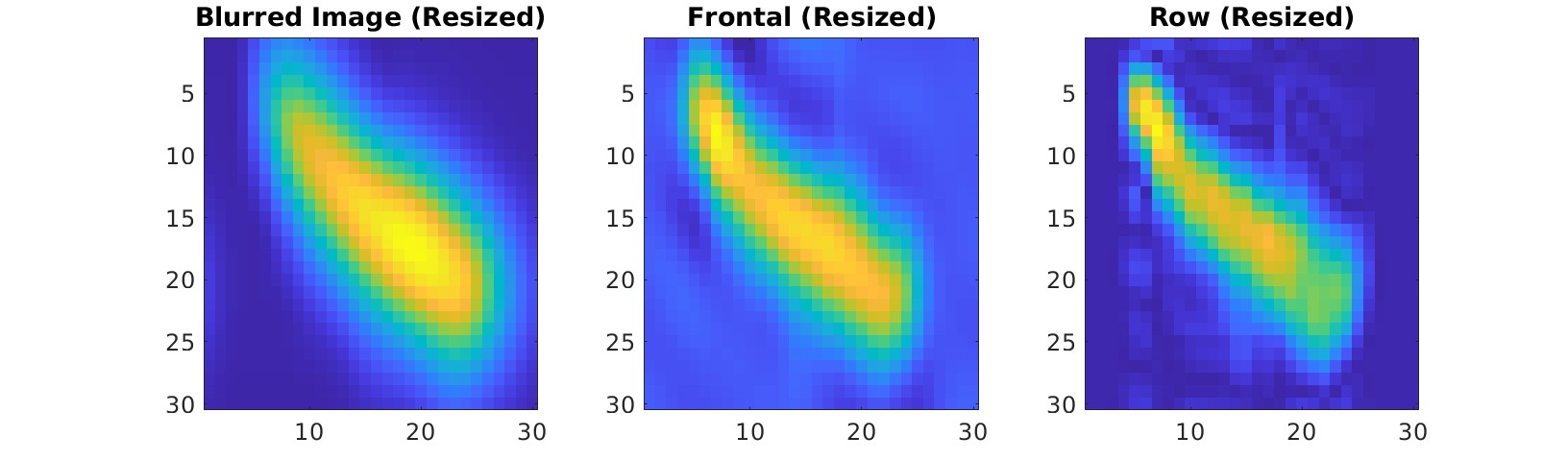}\includegraphics[width=0.35\textwidth]{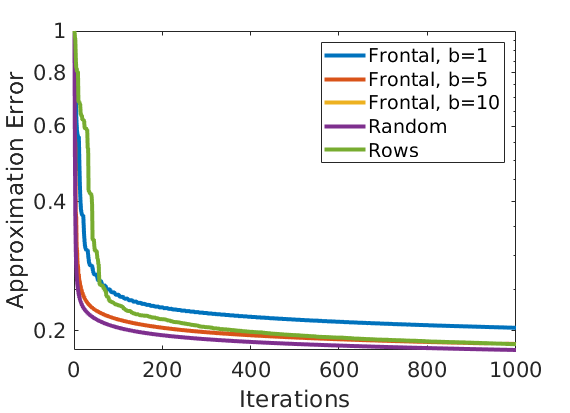}

\begin{tabular}{ccccc}
\toprule 
Method  & MSE  & PSNR  & SSIM  & Time (seconds)\tabularnewline
\midrule 
Cyclic (block size = 1)  & 4.127e-06  & 5.384e+01  & 9.953e-01  & 12.898\tabularnewline
\midrule 
Cyclic (block size = 5)  & 3.468e-06  & 5.460e+01  & 9.962e-01  & 12.359\tabularnewline
\midrule 
Cyclic (block size = 10)  & \textbf{3.250e-06 }  & \textbf{5.488e+01 }  & \textbf{9.965e-01 }  & 12.416\tabularnewline
\midrule 
Random  & \textbf{3.250e-06 }  & \textbf{5.488e+01 }  & \textbf{9.965e-01 }  & 12.521\tabularnewline
\midrule 
TRK  & 3.406e-06  & 5.468e+01  & 9.961e-01  & 1.206\tabularnewline
\bottomrule
\end{tabular}

\caption{\label{fig:experiments_HST} HST Satellite Image from \cite{reichel2022tensor}.
A frontal slice (i.e., frame) from the resized ground truth video
$\mathcal{X}[:,:,15]\in\mathbb{R}^{60\times60\times1}$, the blurred
video $\mathcal{B}[:,:,15]\in\mathbb{R}^{60\times60\times1}$, the
recovered slice when using cyclic frontal slice descent (Algorithm~\ref{alg:tsolve_cyclic_slices}),
random frontal slice descent (Algorithm~\ref{alg:tsolve_random_slices}),
and when using TRK \cite{ma2022randomized}. In the leftmost panel,
we show the final approximation error $\mathcal{E}(T)$ against iterations
with learning rate $\alpha=0.01$ and the maximal number of iterations
of 1000 steps.}
\end{figure}

Figure \ref{fig:experiments_HST} compares the performance of Algorithm~\ref{alg:tsolve_cyclic_slices}
and TRK on image deblurring. The image is a resized version of the
$512\times512$ pixel image of the Hubble Space Telescope in Figure
\ref{fig:hst}. The measurements $\mathcal{B}=\mathcal{A}*\mathcal{X}$
where $\mathcal{A}$ is the $151\times 151\times151$ blurring operator
with $B=9$ and $\sigma=3$. The learning rate is selected to be $\alpha=0.01$.
We put the initial image with double precision where the image is
extended into a 3-way tensor by creating several slices (frames) with
slight variations introduced by Gaussian additive noise. 
\newpage
\begin{figure}[H]
\centering

\textbf{151 by 31 by 151 full-size trimmed video}

\centering
\includegraphics[width=0.5\textwidth,height=3.5cm]{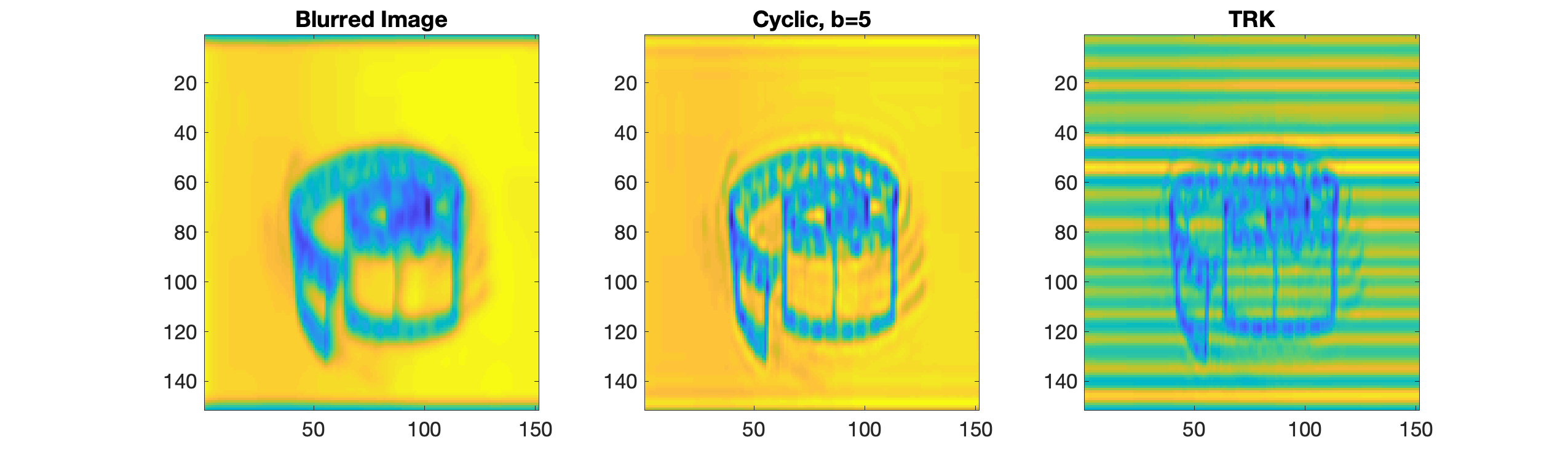}
\includegraphics[width=0.35\textwidth]{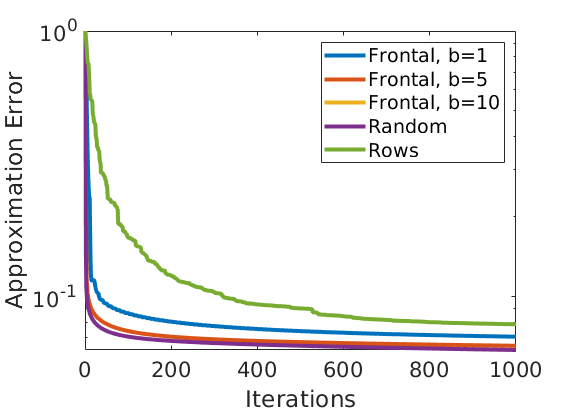}

\begin{tabular}{ccccc}
\toprule 
Method  & MSE  & PSNR  & SSIM  & Time (seconds)\tabularnewline
\midrule 
Cyclic (block size = 1)  & 1.091e-08  & 7.962e+01  & 1.000e+00  & 24.238\tabularnewline
\midrule 
Cyclic (block size = 5)  & 8.400e-09  & 8.076e+01  & 1.000e+00  & 23.965\tabularnewline
\midrule 
Cyclic (block size = 10)  & \textbf{7.789e-09}  & \textbf{8.109e+01}  & \textbf{1.000e+00}  & 26.500\tabularnewline
\midrule 
Random  & \textbf{7.789e-09}  & \textbf{8.109e+01}  & \textbf{1.000e+00}  & 26.048\tabularnewline
\midrule 
TRK  & 1.905e-08  & 7.720e+01  & 9.999e-01  & 2.303\tabularnewline
\bottomrule
\end{tabular}

\caption{\label{fig:experiments_Trimmed} Trimmed video tensor. Top: A frontal
slice (i.e., frame) from the resized ground truth video $\mathcal{X}[:,:,15]\in\mathbb{R}^{60\times60\times1}$,
the blurred video $\mathcal{B}[:,:,15]\in\mathbb{R}^{60\times60\times1}$,
the recovered slice when using cyclic frontal slice descent (Algorithm~\ref{alg:tsolve_cyclic_slices}),
random frontal slice descent (Algorithm~\ref{alg:tsolve_random_slices}),
and when using TRK \cite{tang2022sketch}. In the leftmost panel,
we show the final approximation error $\mathcal{E}(T)$ against iterations
with learning rate $\alpha=0.01$ and the maximal number of iterations
of 1000 steps. Bottom: A frontal slice (i.e., frame) from the ground
truth video $\mathcal{X}[:,:,15]\in\mathbb{R}^{151\times151\times1}$,
the blurred video $\mathcal{B}[:,:,15]\in\mathbb{R}^{151\times151\times1}$,
the recovered slice when using CFGD and TRK. In the leftmost panel,
we show the final approximation error $\mathcal{E}(T)$ against iterations
with learning rate $\alpha=0.001$ and the maximal number of iterations
of 200 steps.}
\end{figure}

In Figure \ref{fig:experiments_Trimmed}, we apply TRK and Algorithm~\ref{alg:tsolve_cyclic_slices}
on video data containing frames from the 1929 film ``Finding His
Voice"~\cite{videodata}. The video is the underlying tensor $\mathcal{X}$
with 31 video frames, each of size $151\times151$. The measurements
$\mathcal{B}=\mathcal{A}*\mathcal{X}$ where $\mathcal{A}$ is the
$151\times151\times151$ blurring operator with $B=9$ and $\sigma=4$.
The learning rate is selected to be $\alpha=0.001$. While both suffer
from the ill-conditioned nature of $\mathcal{A}$, Algorithm~\ref{alg:tsolve_cyclic_slices}
can attain a smaller error faster than TRK.

In addition, the visual quality of the recovery of the proposed method
is better than that of TRK, which can be seen in the convergence plot
in Figure~\ref{fig:experiments_Trimmed}. Figure~\ref{fig:experiments_Trimmed}
presents the original 15th video frame $\mathcal{X}[:,:,15]$, the
blurred video frame, $\mathcal{B}[:,:,15]$, the recovered slice when
using Algorithm~\ref{alg:tsolve_cyclic_slices}, and when using TRK.
The goal is to recover the original image, represented as a tensor
$\mathcal{X}$. The blurred image is $\mathcal{B}$, and the blurring
operation is represented by the tensor $\mathcal{A}$. The t-product
formulation naturally models this convolutional process, allowing
for efficient algorithms to deblur the image by solving \eqref{eq:tensorlinsys}
given known $\mathcal{A}$ and $\mathcal{B}$. The convolutional nature
of the t-product effectively captures the blurring effect and facilitates
the subsequent reconstruction of the original image through iterative
methods that utilize the convolutional structure.

\section{\label{sec:Future-works}Conclusion and Future works}

Our exploration of tensor sketching methods has highlighted the potential
of frontal sketching, but comparable analyses with row and column
sketching methods remain less developed. A crucial area for future
research involves establishing a robust criterion that would allow
for the selection between frontal, row, and column sketching methods
based on specific data characteristics or performance metrics. 

The primary goal will be to develop and validate a set of criteria
or metrics that can effectively determine the optimal sketching technique
for a given tensor dataset and application. These criteria may include,
but are not limited to, considerations of computational efficiency,
memory usage, data sparsity, and the preservation of tensor properties
critical to the application domain. Besides finding a necessary condition for the learning rate in our sketching algorithms, another goal is integrating adaptive
learning rates $\alpha$ (which varies as iterations proceeds) into
the dynamic sketching algorithm, where not only the sketching method
is chosen adaptively, but also the learning rates are adjusted based on real-time feedback from the sketching process, which has an intrinsic link to sparse likelihood approximation when certain kind of sparsity is encoded in the system \cite{hrluo_2019a,LHL2024,chi2012tensors}.

Building on the established criteria, the next phase will involve
designing an adaptive sketching algorithm. This algorithm will dynamically
switch between frontal, row, and column sketching based on real-time
analysis of the data's evolving characteristics during processing.
Such adaptability could significantly enhance performance in diverse
scenarios, ranging from real-time data streaming to large-scale data
processing where different sketching methods may be optimal at different
processing stages.

Although higher-order t-products can be defined in a recursive manner,
extending the sketching methods in the current work for higher-order
tensor t-products generate new challenges. The complexity of defining
and computing such products increases significantly with the tensor
order, thus computational complexity and numerical stability becomes a critical issue and careful consideration of algorithmic design and tensor operations are needed.
In addition, the t-product leverages the block circulant matrix representation
of tensors, where each frontal slice of a tensor is shifted and concatenated
to form a larger block circulant matrix, which can be seen as a form
of circular convolution applied along the third dimension (frontal
slices) of the tensors. As shown by \cite{reichel2022tensor}, the
natural convolution provided by the t-product between tensors opens
a new direction of representing graph-based models, and the sketching
developed for frontal dimensions allows more efficient computations
over them.

It is of interest to explore the application of t-product in compressible tensor systems, for example, the butterfly compression in vector systems have been proven to scale up well \cite{liu2021butterfly,liu2021sparse,kielstra2023tensor} shows promising results for Kronecker tensor system, it remains to explore how tensor compressions work along with other kind of tensor products and associated systems.

\section*{Acknowledgment}

HL was supported by the Director, Office of Science, of the U.S. Department
of Energy under Contract DE-AC02-05CH11231; U.S. National Science
Foundation NSF-DMS 2412403.
We also express our thanks to Dr. Yin-Ting Liao for earlier discussions in this work.

 \bibliographystyle{plain}
\bibliography{TRGS_refs}

\newpage{}

\appendix

\section{Notation Conventions}

We present the notations we used for different mathematical objects. 
\begin{center}
\begin{tabular}{|c|c|}
\hline 
Definition  & Notation\tabularnewline
\hline 
\hline 
Left coefficient tensor  & $\mathcal{A}\in\mathbb{R}^{n_{1}\times n_{2}\times n}$\tabularnewline
\hline 
Right coefficient tensor  & $\mathcal{B}\in\mathbb{R}^{n_{1}\times n_{3}\times n}$\tabularnewline
\hline 
Solution tensor  & $\mathcal{X}\in\mathbb{R}^{n_{2}\times n_{3}\times n}$\tabularnewline
\hline 
True solution tensor  & $\mathcal{X}_{*}\in\mathbb{R}^{n_{2}\times n_{3}\times n}$\tabularnewline
\hline 
Row/column/frontal slice  & $\mathcal{A}_{i::}=\mathcal{A}_{i}\in\mathbb{R}^{n_{1}\times n_{2}\times1}$\tabularnewline
\hline 
Padded slice  & $\tilde{\mathcal{A}}_{i}$ $\in\mathbb{R}^{n_{1}\times n_{2}\times n}$\tabularnewline
\hline 
Solution tensor at iteration $t$  & $\mathcal{X}(t)$\tabularnewline
\hline 
Residual tensor at iteration $t$  & $\mathcal{R}(t)$\tabularnewline
\hline 
Error scalar at iteration $t$  & $\mathcal{E}(t)$\tabularnewline
\hline 
\end{tabular}
\par\end{center}

\section{Proof of Theorem \ref{thm:conv-cyc-n2}: frontal size $n=2$}

\label{sec:cyc-n2} We first consider the case where $n=2$. 
\begin{proof}
In this case, the algorithm cycles through the 2 frontal slices one
at a time, and the update in Algorithm~\ref{alg:tsolve_cyclic_slices}
can then be simplified into the following cases:

Iteration $t=2k$, the slice $\tilde{\mathcal{A}}_{2}$ is chosen
and we have 
\begin{align}
\mathcal{R}(2k) & =\mathcal{B}-\tilde{\mathcal{A}}_{1}*\mathcal{X}(2k-2)-\tilde{\mathcal{A}}_{2}*\mathcal{X}(2k-1),\label{eq:rupdate2_even}\\
\mathcal{X}(2k) & =\mathcal{X}(2k-1)+\alpha\tilde{\mathcal{A}}_{2}^{T}*\mathcal{R}(2k)\nonumber \\
 & =\left(\mathcal{I}-\alpha\tilde{\mathcal{A}}_{2}^{T}*\tilde{\mathcal{A}}_{2}\right)*\mathcal{X}(2k-1)-\alpha\tilde{\mathcal{A}}_{2}^{T}*\tilde{\mathcal{A}}_{1}*\mathcal{X}(2k-2)+\alpha\tilde{\mathcal{A}}_{2}^{T}*\mathcal{B}.\label{eq:xupdate2_even}
\end{align}

Iteration $t=2k+1$, the slice $\tilde{\mathcal{A}}_{1}$ is chosen
and we obtain 
\begin{align}
\mathcal{R}(2k+1) & =\mathcal{B}-\tilde{\mathcal{A}}_{1}*\mathcal{X}(2k)-\tilde{\mathcal{A}}_{2}*\mathcal{X}(2k-1)\label{eq:rupdate2_odd}\\
\mathcal{X}(2k+1) & =\mathcal{X}(2k)+\alpha\tilde{\mathcal{A}}_{1}^{T}*\mathcal{R}(2k+1)\nonumber \\
 & =\left(\mathcal{I}-\alpha\tilde{\mathcal{A}}_{1}^{T}*\tilde{\mathcal{A}}_{1}\right)*\mathcal{X}(2k)-\alpha\tilde{\mathcal{A}}_{1}^{T}*\tilde{\mathcal{A}}_{2}*\mathcal{X}(2k-1)+\alpha\tilde{\mathcal{A}}_{1}^{T}*\mathcal{B}\label{eq:xupdate2_odd}
\end{align}

Let the (scalar) error between the current solution $\mathcal{X}(k)$
and the actual solution $\mathcal{X}_{*}$ at iteration $k$ be $\mathcal{E}(k):=\Vert\mathcal{X}(k)-\mathcal{X}_{*}\Vert$.
Recall $\mathcal{X}_{*}$ satisfies $(\tilde{\mathcal{A}}_{1}+\tilde{\mathcal{A}}_{2})*\mathcal{X}_{*}=\mathcal{A}*\mathcal{X}_{*}=\mathcal{B}$.
According to the two recursive relations in \eqref{eq:xupdate2_even}
and \eqref{eq:xupdate2_odd}, we recognize that \eqref{eq:xupdate2_even}
can be used to compute $\mathcal{E}(2k)$ 
\begin{align}
\mathcal{E}(2k) & =\left\Vert \mathcal{X}(2k)-\mathcal{X}_{*}\right\Vert \nonumber \\
 & =\left\Vert \left(\mathcal{I}-\alpha\tilde{\mathcal{A}}_{2}^{T}*\tilde{\mathcal{A}}_{2}\right)*\mathcal{X}(2k-1)-\alpha\tilde{\mathcal{A}}_{2}^{T}*\tilde{\mathcal{A}}_{1}*\mathcal{X}(2k-2)+\alpha\tilde{\mathcal{A}}_{2}^{T}*\left(\tilde{\mathcal{A}}_{1}+\tilde{\mathcal{A}}_{2}\right)*\mathcal{X}_{*}-\mathcal{X}_{*}\right\Vert \nonumber \\
 & =\left\Vert \left(\mathcal{I}-\alpha\tilde{\mathcal{A}}_{2}^{T}*\tilde{\mathcal{A}}_{2}\right)*\left(\mathcal{X}(2k-1)-\mathcal{X}_{*}\right)-\alpha\tilde{\mathcal{A}}_{2}^{T}*\tilde{\mathcal{A}}_{1}\left(\mathcal{X}(2k-2)-\mathcal{X}_{*}\right)\right\Vert \nonumber \\
 & \leq\left\Vert \mathcal{I}-\alpha\tilde{\mathcal{A}}_{2}^{T}*\tilde{\mathcal{A}}_{2}\right\Vert _{op}\left\Vert \mathcal{X}(2k-1)-\mathcal{X}_{*}\right\Vert +\left\Vert \alpha\tilde{\mathcal{A}}_{2}^{T}*\tilde{\mathcal{A}}_{1}\right\Vert _{op}\left\Vert \mathcal{X}(2k-2)-\mathcal{X}_{*}\right\Vert \nonumber \\
 & =\left\Vert \mathcal{I}-\alpha\tilde{\mathcal{A}}_{2}^{T}*\tilde{\mathcal{A}}_{2}\right\Vert _{op}\mathcal{E}(2k-1)+\left\Vert \alpha\tilde{\mathcal{A}}_{2}^{T}*\tilde{\mathcal{A}}_{1}\right\Vert _{op}\mathcal{E}(2k-2),\label{eq:recursive1}
\end{align}
and similarly we can use \eqref{eq:xupdate2_odd} for $\mathcal{E}(2k+1)$,
to yield 
\begin{align}
\mathcal{E}(2k+1)\leq\left\Vert \mathcal{I}-\alpha\tilde{\mathcal{A}}_{1}^{T}*\tilde{\mathcal{A}}_{1}\right\Vert _{op}\mathcal{E}(2k)+\left\Vert \alpha\tilde{\mathcal{A}}_{1}^{T}*\tilde{\mathcal{A}}_{2}\right\Vert _{op}\mathcal{E}(2k-1).\label{eq:recursive2}
\end{align}
Then we can combine \eqref{eq:recursive1} and \eqref{eq:recursive2}
reduce to 
\begin{align}
 & \mathcal{E}(2k)\leq\left\Vert \mathcal{I}-\alpha\tilde{\mathcal{A}}_{2}^{T}*\tilde{\mathcal{A}}_{2}\right\Vert _{op}\mathcal{E}(2k-1)+\left\Vert \alpha\tilde{\mathcal{A}}_{2}^{T}*\tilde{\mathcal{A}}_{1}\right\Vert _{op}\mathcal{E}(2k-2),\nonumber \\
 & \mathcal{E}(2k+1)\leq\left\Vert \mathcal{I}-\alpha\tilde{\mathcal{A}}_{1}^{T}*\tilde{\mathcal{A}}_{1}\right\Vert _{op}\mathcal{E}(2k)+\left\Vert \alpha\tilde{\mathcal{A}}_{1}^{T}*\tilde{\mathcal{A}}_{2}\right\Vert _{op}\mathcal{E}(2k-1).\label{eq:recursive_simplified}
\end{align}
Now let $\kappa_{1}=\|\mathcal{I}-\alpha\tilde{\mathcal{A}}_{2}^{T}*\tilde{\mathcal{A}}_{2}\|_{op}$,
$\kappa_{2}=\|\mathcal{I}-\alpha\tilde{\mathcal{A}}_{2}^{T}*\tilde{\mathcal{A}}_{2}\|_{op}$,
and $\mu=\|\tilde{\mathcal{A}}_{1}^{T}*\tilde{\mathcal{A}}_{2}\|_{op}$.
We can pick $\alpha$ so that $\kappa_{1},\kappa_{2}<1$. Let the
union bound of errors of the consecutive iterations $\mathcal{M}_{k}\coloneqq\max\left\{ \|\mathcal{E}(2k-1)\|,\|\mathcal{E}(2k-2)\|\right\} $.
Then \eqref{eq:recursive1} and \eqref{eq:recursive2} reduce to 
\begin{align*}
\mathcal{E}(2k) & \leq\kappa_{2}\mathcal{E}(2k-1)+\alpha\mu\mathcal{E}(2k-2),\\
 & =\left(\kappa_{2}+\alpha\mu\right)\mathcal{M}_{k},\\
\mathcal{E}(2k+1) & \leq\kappa_{1}\mathcal{E}(2k)+\alpha\mu\mathcal{E}(2k-1)\\
 & \leq\kappa_{1}\left(\kappa_{2}+\alpha\mu\right)\mathcal{M}_{k}+\alpha\mu\mathcal{E}(2k-1)\\
 & \leq(\kappa_{1}\kappa_{2}+(1+\kappa_{1})\alpha\mu)\mathcal{M}_{k},
\end{align*}
and thus, 
\[
\mathcal{M}_{k+1}\leq C\mathcal{M}_{k},
\]
where $C=\max\{\kappa_{2}+\alpha\mu,\kappa_{1}\kappa_{2}+(1+\kappa_{1})\alpha\mu\}$
and if $\alpha<\min\{\frac{1-\kappa_{2}}{\mu},\frac{1-\kappa_{1}\kappa_{2}}{\mu(1+\kappa_{2})}\}$
then $C<1$. Namely, we proved the result. 
\end{proof}
\begin{rem}
From the above proof, we know that the convergence rate is controlled
by the correlation $\mu$ between slices defined by t-product, and
the learning rate $\alpha$.

Note that if frontal slices are mutually orthogonal, e.g., if we have
a diagonal tensor $\mathcal{A}$, then $\mu=0$ and 
\begin{align*}
\mathcal{E}(2k) & \leq\kappa_{2}\mathcal{E}(2k-1)\\
\mathcal{E}(2k+1) & \leq\kappa_{1}\mathcal{E}(2k)\leq\kappa_{1}^{k}\kappa_{2}^{k}\mathcal{E}(0).
\end{align*}
Since $\kappa_{1},\kappa_{2}<1$, this implies exponential convergence
of Cyclic Slice Descent (Algorithm \ref{alg:tsolve_cyclic_slices})
in the form of $\mathcal{E}(t)\leq\kappa^{t+1}\mathcal{E}(0)$ with
factor $\kappa$ where $\kappa\coloneqq\max_{j=1,2}\|\mathcal{I}-\alpha\mathcal{\tilde{A}}_{j}^{T}*\tilde{\mathcal{A}}_{j}\|_{op}$
for $n=2{\color{orange}}$. 

\end{rem}

\section{Proof of Theorem~\ref{thm:conv-cyclic}}
\label{subsec:conv-cyclic}
Following the idea of proof in Appendix
\ref{sec:cyc-n2}, we extend to the case where $n>2$ for Algorithm
\ref{alg:tsolve_cyclic_slices}. We list the assumptions for our reference:

To prove Theorem~\ref{thm:conv-cyclic}, we will show that the upper bound on the approximation error decreases with every iteration. Lemma~\ref{lem:error} provides a relationship between the current and previous $n$ approximation errors. This is used in Lemma~\ref{lem:contract} to define the bounding coefficient. Lemma~\ref{lem:decrease} shows that the bounding coefficient is decreasing. Before we prove these lemmas and the main theorem, we remind the reader of useful notation and assumptions in this section. In particular, we let
\begin{align}
\kappa & =\max_{i=1,...n}\|\mathcal{I}-\alpha\tilde{\mathcal{A}}_{i}^{T}*\tilde{\mathcal{A}}_{i}\|_{op},\label{eq:def_kappa_mu}\\
\mu & =\max_{\substack{i,j=1,...,n\\
i\neq j
}
}\|\tilde{\mathcal{A}}_{i}^{T}*\tilde{\mathcal{A}}_{j}\|_{op},\nonumber 
\end{align}
and $\|\cdot\|_{op}$
is the operator norm, defined as $\|\mathcal{A}\|_{op}\coloneqq\sup_{\|\mathcal{X}\|=1}\|\mathcal{A}*\mathcal{X}\|_{2}$. We adopt the convention that $\mathcal{X}(t) = 0$ when $t < 0$ and thus, $\mathcal{E}(t) = \|\mathcal{X}^* \|_F$ when $t < 0$.

\begin{assum} We assume the learning rate $\alpha$ is chosen such that $\kappa < 1$ and 
$$ \kappa + \alpha \mu (n-1)<1.$$
    \label{assum:alpha}
\end{assum}

\begin{lem} The approximation error at iteration $t+1$ is bounded by a function of error from the previous $n$ iterations:
    \begin{equation}
       \mathcal{E}(t+1)  \leq\kappa\mathcal{E}(t)+\alpha\mu\sum_{i=1}^{n-1}\mathcal{E}(t-i).\label{eq:iter_bound}
    \end{equation}
    \label{lem:error}
\end{lem}

\begin{proof}
Recall that the iteration step in Algorithm \ref{alg:tsolve_cyclic_slices}:
\begin{equation*}
\mathcal{X}(t+1)=\mathcal{X}(t)+\alpha\tilde{A}_{[t\Mod{n}]+1}^{T}*\mathcal{R}(t+1),
\end{equation*}
such that, since $\mathcal{X}(t) = 0$ for $t<0$, we can consider the last $n$ iterated slices.
\begin{align}
 \mathcal{R}(t+1)&=\mathcal{R}(t)-\tilde{\mathcal{A}}_{i}*\mathcal{X}(t)+\tilde{\mathcal{A}}_{i}*\mathcal{X}(t-n+1) \nonumber \\
 & =\mathcal{B}-\left( \tilde{\mathcal{A}}_{t\Mod n+1}*\mathcal{X}(t)+\tilde{\mathcal{A}}_{(t-1)\Mod n+1}*\mathcal{X}(t-1)+\right. \nonumber\\
 & \left. \cdots+\tilde{\mathcal{A}}_{(t-n+1)\mod n+1}*\mathcal{X}(t-n+1)\right)\nonumber \\
 & =\mathcal{B}-\sum_{i=0}^{n-1}\tilde{\mathcal{A}}_{[(t-i)\Mod n]+1}*\mathcal{X}(t-i)\nonumber \\
 &=\sum_{i=0}^{n-1}\tilde{\mathcal{A}}_{[(t-i)\Mod n]+1}*\left(\mathcal{X}_{*}-\mathcal{X}(t-i)\right).\label{eq:residual_recursion_general}
\end{align}
The error $\mathcal{E}(t+1)$ can be bounded
as: 
\begin{align}
 \mathcal{E}(t+1)\nonumber & =\left\Vert \mathcal{X}(t+1)-\mathcal{X}_{*}\right\Vert _{F}\nonumber \\
 & =\left\Vert \mathcal{X}(t)+\alpha\tilde{\mathcal{A}}_{\left[t\Mod n\right]+1}^{T}*\mathcal{R}(t+1)-\mathcal{X}_{*}\right\Vert _{F}\label{eq:loose_bound} \\
 & \leq\left\Vert \left(\mathcal{I}-\alpha\tilde{\mathcal{A}}_{\left[t\Mod n\right]+1}^{T}*\tilde{\mathcal{A}}_{\left[t\Mod n\right]+1}\right)*\left(\mathcal{X}(t)-\mathcal{X}_{*}\right)\right\Vert _{F}\nonumber \\
 & \quad\quad+\left\Vert \alpha\tilde{\mathcal{A}}_{\left[t\Mod n\right]+1}^{T}*\left(\sum_{i=1}^{n-1}\tilde{\mathcal{A}}_{\left[(t-i)\Mod n\right]+1}*\left(\mathcal{X}_{*}-\mathcal{X}(t-i)\right)\right)\right\Vert _{F}\\
 & \leq\left\Vert \mathcal{I}-\alpha\tilde{\mathcal{A}}_{\left[t\Mod n\right]+1}^{T}*\tilde{\mathcal{A}}_{\left[t\Mod n\right]+1}\right\Vert _{op}\left\Vert \left(\mathcal{X}_{*}-\mathcal{X}(t)\right)\right\Vert _{F}\nonumber \\
 & \quad\quad +\alpha\sum_{i=1}^{n-1}\left\Vert \tilde{\mathcal{A}}_{\left[t\Mod n\right]+1}^{T}*\tilde{\mathcal{A}}_{\left[(t-i)\Mod n\right]+1}\right\Vert _{op}\left\Vert \left(\mathcal{X}_{*}-\mathcal{X}(t-i)\right)\right\Vert _{F} \label{eq:loose_bound2}\\
 & =\left\Vert \mathcal{I}-\alpha\tilde{\mathcal{A}}_{\left[t\Mod n\right]+1}^{T}*\tilde{\mathcal{A}}_{\left[t\Mod n\right]+1}\right\Vert _{op}\mathcal{E}(t)\nonumber\\
 & +\alpha\sum_{i=1}^{n-1}\left\Vert \tilde{\mathcal{A}}_{\left[t\Mod n\right]+1}^{T}*\tilde{\mathcal{A}}_{\left[(t-i)\Mod n\right]+1}\right\Vert _{op}\mathcal{E}(t-i) \\
 & \leq\kappa\mathcal{E}(t)+\alpha\mu\sum_{i=1}^{n-1}\mathcal{E}(t-i),\label{eq:iter_bound}
\end{align}
where the first inequality follows from \eqref{eq:residual_recursion_general} and the triangle inequality, the second inequality follows from Remark~\ref{rem:op-fro}, and then the last inequality uses the definitions of $\kappa$ and $\mu$.

In \eqref{eq:loose_bound} and \eqref{eq:loose_bound2}, we have rather conservative bounds from these inequalities. Although there are scenarios where the equalities hold in these two inequalities, it is possible to improve these bounds.
\end{proof}

\begin{lem} Define $\epsilon_{t} \in \mathbb{R} $ as 
\begin{equation}
\epsilon_{t} := 
    \begin{cases}
\kappa^{t} + \alpha \mu \sum_{j=0}^{t-1} \kappa^{j} \sum_{i=1}^{n-1} \epsilon_{t-1-j-i} & \text{if } t \geq 1 \\ 
1 & \text{otherwise.}
\end{cases}
\label{eq:eps}
\end{equation}
For all $t \geq 1$, the approximation error at iteration $t$ is bounded above by: 
\begin{equation}
\mathcal{E}(t) \leq \epsilon_{t} \mathcal{E}(0).
    \label{eq:contract}
\end{equation}
    \label{lem:contract}
\end{lem}

\begin{proof}
When $t \leq 0$ 
\begin{equation*}
    \mathcal{E}(t) = \| \mathcal{X}(t) - \mathcal{X}_*\| = \mathcal{E}(0),
\end{equation*}
since $\mathcal{X}(t) = \mathcal{X}(0) = 0$ so $\epsilon_t = 1$.
    For $t \geq 1$, we proceed by induction. In the base case, we first consider $t=1$. By \eqref{eq:eps}, when $t=1$:
    \begin{align*}
        \epsilon_1 &= \kappa^{1} + \alpha \mu \sum_{j=0}^{0} \kappa^{j} \sum_{i=1}^{n-1} \epsilon_{1-1-j-i} \\ &= \kappa + \alpha \mu \sum_{i=1}^{n-1} 1 = \kappa + \alpha \mu (n-1).
    \end{align*}
    On the other hand, by Lemma~\ref{lem:error}, we have 
    \begin{align*}
        \mathcal{E}(1) &\leq \kappa \mathcal{E}(0) + \alpha \mu \sum_{i=1}^{n-1} \mathcal{E}(0-i) \\ 
        & = \left( \kappa + \alpha \mu (n-1) \right) \mathcal{E}(0) = \epsilon_1 \mathcal{E}(0).
    \end{align*}
    where the equality follows from the convention that $\mathcal{E}(i) = \mathcal{E}(0)$ for $i<0$. Thus, $\mathcal{E}(1) \leq \epsilon_1 \mathcal{E}(0).$
    Now let $T \in \mathbb{N}$ and assume that for all $t \leq T$, \eqref{eq:contract} holds. Again, by Lemma~\ref{lem:error}, we have
    \begin{align*}
        \mathcal{E}(T+1) &\leq \kappa \mathcal{E}(T) + \alpha \mu \sum_{i=1}^{n-1} \mathcal{E}(T-i) \\ 
        &\leq \left(\kappa \epsilon_T + \alpha \mu \sum_{i=1}^{n-1} \epsilon_{(T-i)}\right) \mathcal{E}(0),
    \end{align*}
    where the second inequality follows from our inductive hypothesis. Using the definition of $\epsilon_T$ and simplifying terms, we arrive at our desired conclusion: 
\begin{align*}
\mathcal{E}(T+1) & \leq\left(\kappa\epsilon_{T}+\alpha\mu\sum_{i=1}^{n-1}\epsilon_{(T-i)}\right)\mathcal{E}(0)\\
 & =\left(\kappa\left(\kappa^{T}+\alpha\mu\sum_{j=0}^{T-1}\kappa^{j}\sum_{i=1}^{n-1}\epsilon_{T-1-j-i}\right)+\alpha\mu\sum_{i=1}^{n-1}\epsilon_{(T-i)}\right)\mathcal{E}(0)\\
 \end{align*}
\begin{align*}
 & =\left(\kappa^{T+1}+\alpha\mu\sum_{j=0}^{T-1}\kappa^{j+1}\sum_{i=1}^{n-1}\epsilon_{T-1-j-i}+\alpha\mu\sum_{i=1}^{n-1}\epsilon_{(T-i)}\right)\mathcal{E}(0)\\
& =\left(\kappa^{T+1}+\alpha\mu\sum_{j=1}^{T}\kappa^{j}\sum_{i=1}^{n-1}\epsilon_{T-j-i}+\alpha\mu\cdot\kappa^{0}\sum_{i=1}^{n-1}\epsilon_{(T-i)}\right)\mathcal{E}(0)\\
 & =\left(\kappa^{T+1}+\alpha\mu\sum_{j=0}^{T}\kappa^{j}\sum_{i=1}^{n-1}\epsilon_{T-j-i}\right)\mathcal{E}(0)\\
 & =\epsilon_{T+1}\mathcal{E}(0).
\end{align*}
Note that we re-indexing in line 4 of the above sequence of arguments.
\end{proof}

\begin{lem} Let $\epsilon_t$ be as defined in \eqref{eq:eps}. For all $t \geq 0$, $\epsilon_{t+1} < \epsilon_t$.
    \label{lem:decrease}
\end{lem}
\begin{proof}
    We proceed by induction. For the base case, let $t = 0$ then by definition
    \begin{align*}
        \epsilon_1 = \left( \kappa + \alpha \mu (n-1) \right) \epsilon_0 < \epsilon_0,
    \end{align*}
    where the inequality follows from the second inequality $\kappa+\alpha\mu(n-1)<1$ of  Assumption~\ref{assum:alpha}.
    Now let $T \in \mathbb{N}$ and suppose for all $t \leq T$, $\epsilon_{t+1} < \epsilon_{t}$ holds. Then for $\epsilon_{(T+1)+1}$, we have by definition of $\epsilon_t$:
    \begin{align*}
       \epsilon_{(T+1)+1} &= \kappa^{(T+1) + 1} + \alpha \mu \sum_{j=0}^{T+1} \kappa^{j} \sum_{i=1}^{n-1} \epsilon_{T+1-j-i} \\
      &= \kappa^{(T+1) + 1}  + \alpha \mu \sum_{j=0}^{T} \kappa^{j} \sum_{i=1}^{n-1} \epsilon_{T+1-j-i} +\alpha \mu \kappa^{T+1} \sum_{i=1}^{n-1} \epsilon_{T+1-(T+1)-i} \\
      &= \kappa^{(T+1)} \left(\kappa + \alpha \mu (n-1) \right) +\sum_{j=0}^{T} \kappa^{j} \sum_{i=1}^{n-1} \epsilon_{T+1-j-i} \\
      & < \kappa^{(T+1)} \left(\kappa + \alpha \mu (n-1) \right) +\sum_{j=0}^{T} \kappa^{j} \sum_{i=1}^{n-1} \epsilon_{T-j-i} \\
      & < \kappa^{(T+1)}  + \alpha \mu \sum_{j=0}^{T} \kappa^{j} \sum_{i=1}^{n-1} \epsilon_{T-j-i} = \epsilon_{T+1},
    \end{align*}
    where the first inequality follows from the inductive hypothesis and the second inequality follows from the second inequality $\kappa+\alpha\mu<1$ of Assumption~\ref{assum:alpha}. 
\end{proof}

\begin{proof} (Proof of Theorem~\ref{thm:conv-cyclic}) Suppose that Assumption~\ref{assum:alpha} holds. By Lemma~\ref{lem:contract}, the approximation error at iteration $t$ is $\mathcal{E}(t) \leq \epsilon_t \mathcal{E}(0)$ and by Lemma~\ref{lem:decrease}, $\epsilon_t \geq 0$ is decreasing. Since $\mathcal{E}(t) \geq 0$ for all $t$, this implies that the approximation error is decreasing to zero. 
\end{proof}

\section{Proof of Theorem~\ref{thm:conv-cyclic-noise} }

\label{subsec:conv-cyclic-noise}

Note that the proof of Theorem~\ref{thm:conv-cyclic} relies on the
system being consistent. In Lemma~\ref{lem:error}, we use the fact
that $\mathcal{B}=\mathcal{A}*\mathcal{X}^{*}$. Now we assume that
$\mathcal{B}+\mathcal{B}_{e}=\mathcal{A}*\mathcal{X}^{*}$ where $\mathcal{B}_{e}$
is the tensor error term. We still want to study $\mathcal{R}(t)\coloneqq\left\Vert \mathcal{B}-\mathcal{A}*\mathcal{X}(t)\right\Vert _{F}$
and the recursion is as follows:
\begin{align}
\mathcal{R}(t+1) & =\mathcal{R}(t)-\tilde{\mathcal{A}}_{i}*\mathcal{X}(t)+\tilde{\mathcal{A}}_{i}*\mathcal{X}(t-n+1)\nonumber \\
 & =\mathcal{B}-\mathcal{B}_{e}-\sum_{i=0}^{n-1}\tilde{\mathcal{A}}_{[(t-i)\Mod n]+1}*\mathcal{X}(t-i)\nonumber \\
 & =\sum_{i=0}^{n-1}\tilde{\mathcal{A}}_{[(t-i)\Mod n]+1}*\left(\mathcal{X}_{*}-\mathcal{X}(t-i)\right)-\mathcal{B}_{e}.\label{eq:residual_recursion_general-1}
\end{align}
\begin{align}
\mathcal{E}(t+1) & =\left\Vert \mathcal{X}(t+1)-\mathcal{X}_{*}\right\Vert _{F}\nonumber \\
 & =\left\Vert \mathcal{X}(t)+\alpha\tilde{\mathcal{A}}_{\left[t\Mod n\right]+1}^{T}*\mathcal{R}(t+1)-\mathcal{X}_{*}\right\Vert _{F}\nonumber \\
 & \leq\left\Vert \left(\mathcal{I}-\alpha\tilde{\mathcal{A}}_{\left[t\Mod n\right]+1}^{T}*\tilde{\mathcal{A}}_{\left[t\Mod n\right]+1}\right)*\left(\mathcal{X}(t)-\mathcal{X}_{*}\right)\right\Vert _{F}\nonumber \\
 & \quad\quad+\left\Vert \alpha\tilde{\mathcal{A}}_{\left[t\Mod n\right]+1}^{T}*\mathcal{B}_{e}\right\Vert _{F}\nonumber \\
 & \quad\quad+\left\Vert \alpha\tilde{\mathcal{A}}_{\left[t\Mod n\right]+1}^{T}*\left(\sum_{i=1}^{n-1}\tilde{\mathcal{A}}_{\left[(t-i)\Mod n\right]+1}*\left(\mathcal{X}_{*}-\mathcal{X}(t-i)\right)\right)\right\Vert _{F}\\
 & \leq\left\Vert \mathcal{I}-\alpha\tilde{\mathcal{A}}_{\left[t\Mod n\right]+1}^{T}*\tilde{\mathcal{A}}_{\left[t\Mod n\right]+1}\right\Vert _{op}\left\Vert \left(\mathcal{X}_{*}-\mathcal{X}(t)\right)\right\Vert _{F}\nonumber \\
 & \quad\quad+\left\Vert \alpha\tilde{\mathcal{A}}_{\left[t\Mod n\right]+1}^{T}*\mathcal{B}_{e}\right\Vert _{F}\\
 & \quad\quad+\alpha\sum_{i=1}^{n-1}\left\Vert \tilde{\mathcal{A}}_{\left[t\Mod n\right]+1}^{T}*\tilde{\mathcal{A}}_{\left[(t-i)\Mod n\right]+1}\right\Vert _{op}\left\Vert \left(\mathcal{X}_{*}-\mathcal{X}(t-i)\right)\right\Vert _{F}\nonumber \\
 & =\left\Vert \mathcal{I}-\alpha\tilde{\mathcal{A}}_{\left[t\Mod n\right]+1}^{T}*\tilde{\mathcal{A}}_{\left[t\Mod n\right]+1}\right\Vert _{op}\mathcal{E}(t)+\alpha\left\Vert \tilde{\mathcal{A}}_{\left[t\Mod n\right]+1}^{T}*\mathcal{B}_{e}\right\Vert _{F}+\\
 &\alpha\sum_{i=1}^{n-1}\left\Vert \tilde{\mathcal{A}}_{\left[t\Mod n\right]+1}^{T}*\tilde{\mathcal{A}}_{\left[(t-i)\Mod n\right]+1}\right\Vert _{op}\mathcal{E}(t-i)\nonumber \\
 & \leq\kappa\mathcal{E}(t)+\alpha\mu\sum_{i=1}^{n-1}\mathcal{E}(t-i)+\alpha\eta_{e}\label{eq:iter_bound-3}
\end{align}
\newpage
where $\eta_{e}\coloneqq\max_{i=1,\cdots,n}\left\Vert \tilde{\mathcal{A}}_{i}\right\Vert _{op}\left\Vert \mathcal{B}_{e}\right\Vert $.
Now consider the following definition and assumptions
\begin{equation}
\eta_{t}:=\begin{cases}
\kappa^{t}+\alpha\mu\sum_{j=0}^{t-1}\kappa^{j}\sum_{i=1}^{n-1}\epsilon_{t-1-j-i}+\frac{\alpha}{\mathcal{E}(0)}\eta_{e}=\epsilon_{t}+\frac{\alpha}{\mathcal{E}(0)}\eta_{e} & \text{if }t\geq1\\
1 & \text{otherwise.}
\end{cases}\label{eq:eps-1}
\end{equation}
\begin{assum} We assume the learning rate $\alpha$ is chosen such that $\kappa < 1$ and 
$$ \kappa+\alpha\mu(n-1)+\frac{\alpha}{\mathcal{E}(0)}\eta_{e}<1.$$
    \label{assum:alpha-noise}
\end{assum}

 Then we repeat the arguments in Lemma \ref{lem:error} by supposing
that 
\begin{align*}
\mathcal{E}(T+1) & \leq\kappa\mathcal{E}(T)+\alpha\mu\sum_{i=1}^{n-1}\mathcal{E}(T-i)+\alpha\eta_{e}\\
 & =\left(\kappa\epsilon_{T}+\alpha\mu\sum_{i=1}^{n-1}\epsilon_{T-i}+\frac{\alpha}{\mathcal{E}(0)}\eta_{e}\right)\mathcal{E}(0)\\
 & =\left(\kappa\left(\kappa^{T}+\alpha\mu\sum_{j=0}^{T-1}\kappa^{j}\sum_{i=1}^{n-1}\epsilon_{T-1-j-i}\right)+\alpha\mu\sum_{i=1}^{n-1}\epsilon_{T-i}+\frac{\alpha}{\mathcal{E}(0)}\eta_{e}\right)\mathcal{E}(0)\\
\\
 & =\left(\kappa^{T+1}+\alpha\mu\sum_{j=1}^{T}\kappa^{j}\sum_{i=1}^{n-1}\epsilon_{T-j-i}+\alpha\mu\sum_{i=1}^{n-1}\epsilon_{T-i}+\frac{\alpha}{\mathcal{E}(0)}\eta_{e}\right)\mathcal{E}(0)\\
 & =\left(\kappa^{T+1}+\alpha\mu\sum_{j=0}^{T}\kappa^{j}\sum_{i=1}^{n-1}\epsilon_{T-j-i}+\frac{\alpha}{\mathcal{E}(0)}\eta_{e}\right)\mathcal{E}(0)\\
 & =\eta_{T+1}\mathcal{E}(0).
\end{align*}
 And then $\eta_{t+1}<\eta_{t}$ for $t\geq1$ follows from the same
lines as Lemma \ref{lem:decrease}: as long as $\kappa+\alpha\mu(n-1)+\frac{\alpha}{\mathcal{E}(0)}\eta_{e}<1$,
the line 5 of the below argument goes through:
\begin{align*}
\eta_{(T+1)+1} & =\epsilon_{T+1}+\frac{\alpha}{\mathcal{E}(0)}\eta_{e}\\
 & =\kappa^{(T+1)+1}+\alpha\mu\sum_{j=0}^{T+1}\kappa^{j}\sum_{i=1}^{n-1}\epsilon_{T+1-j-i}+\frac{\alpha}{\mathcal{E}(0)}\eta_{e}\\
 & =\kappa^{(T+1)+1}+\alpha\mu\sum_{j=0}^{T}\kappa^{j}\sum_{i=1}^{n-1}\epsilon_{T+1-j-i}+\alpha\mu\kappa^{T+1}\sum_{i=1}^{n-1}\epsilon_{T+1-(T+1)-i}+\frac{\alpha}{\mathcal{E}(0)}\eta_{e}\\
  \end{align*}
 \begin{align*}
 & =\kappa^{(T+1)}\left(\kappa+\alpha\mu(n-1)\right)+\sum_{j=0}^{T}\kappa^{j}\sum_{i=1}^{n-1}\epsilon_{T+1-j-i}+\frac{\alpha}{\mathcal{E}(0)}\eta_{e}\\
 & <\kappa^{(T+1)}\left(\kappa+\alpha\mu(n-1)+\frac{\alpha}{\mathcal{E}(0)}\eta_{e}\right)+\sum_{j=0}^{T}\kappa^{j}\sum_{i=1}^{n-1}\epsilon_{T-j-i}+\frac{\alpha}{\mathcal{E}(0)}\eta_{e}\\
 & <\kappa^{(T+1)}+\alpha\mu\sum_{j=0}^{T}\kappa^{j}\sum_{i=1}^{n-1}\epsilon_{T-j-i}+\frac{\alpha}{\mathcal{E}(0)}\eta_{e} =\epsilon_{T+1}+\frac{\alpha}{\mathcal{E}(0)}\eta_{e}=\eta_{T+1}.
\end{align*}

\section{ Corollary \ref{cor:conv_block}: blocked
Case}
\label{subsec:Blocked-Case}

Suppose that we still perform cyclic descent, but each time we consider
a block of $s$ consecutive slices. For simplicity we assume $n/s$
is an integer, and this mean that we can write the residual as follows when we consider the last $n/s$ iterated slices. 
\begin{align}
 & \mathcal{R}(t+1)\nonumber \\
 & =\mathcal{B}-\left( \tilde{\mathcal{A}}_{(t\mod(n/s))}^{s}*\mathcal{X}(t)+\tilde{\mathcal{A}}_{(t-1\mod(n/s))}^{s}*\mathcal{X}(t-1)+\right.\nonumber \\
 & \left.\cdots+\tilde{\mathcal{A}}_{(t-n/s\mod(n/s))}^{s}*\mathcal{X}(t-T+1)\right),\nonumber \\
 & =\mathcal{B}-\sum_{i=0}^{n/s-1}\tilde{\mathcal{A}}_{(t-i)}^{s}*\mathcal{X}(t-i)\nonumber \\
 & =\sum_{i=0}^{n/s-1}\tilde{\mathcal{A}}_{(t-i)}^{s}*\left(\mathcal{X}_{*}-\mathcal{X}(t-i)\right).\label{eq:residual_recursion_general-2}
\end{align}
We define $\tilde{\mathcal{A}}_{i}^{s}\coloneqq\tilde{\mathcal{A}}_{i\cdot(s-1)+1}+\tilde{\mathcal{A}}_{i\cdot(s-1)+2}+\cdots+\tilde{\mathcal{A}}_{i\cdot s}$
to be the sum of $s$ padded slices. And the same lines of arguments
in \eqref{eq:residual_recursion_general}  still hold with
$T=n/s$ in \eqref{eq:residual_recursion_general-2}: 
\begin{align}
\mathcal{E}(t+1) & \leq\kappa(s)\mathcal{E}(t)+\alpha\mu\sum_{i=1}^{n/s-1}\mathcal{E}(t-i)\label{eq:iter_bound-2}\\
 & \leq\kappa(s)^{t}\mathcal{E}(0)+\alpha\mu(s)\sum_{j=0}^{t}\sum_{i=1}^{n/s}\kappa(s)^{j}\mathcal{E}(t-j-i),\nonumber 
\end{align}
where we need new constants $\kappa(s),\mu(s)$ depending on the block
size $s$, defined with respect to each block (when $s=1$ these reduces
to \eqref{eq:def_kappa_mu}) 
\begin{align}
\kappa(s) & =\max_{i=1,...n/s}\|\mathcal{I}-\alpha\tilde{\mathcal{A}}_{i}^{sT}*\tilde{\mathcal{A}}_{i}^{s}\|_{op},\\
\mu(s) & =\max_{\substack{i,j=1,...,n/s\\
i\neq j
}
}\|\tilde{\mathcal{A}}_{i}^{sT}*\tilde{\mathcal{A}}_{j}^{s}\|_{op}.\nonumber 
\end{align}
And the analog of the equation \eqref{eq:eps} can be written as
follows, where the contraction factor $\epsilon_{t}(s)$ depends on
the block size $s$: 
\begin{equation}
\epsilon_{t}(s) := 
    \begin{cases}
\kappa(s)^{t} + \alpha \mu(s) \sum_{j=0}^{t-1} \kappa^{j}(s) \sum_{i=1}^{n-1} \epsilon_{t-1-j-i}(s) & \text{if } t \geq 1 \\ 
1 & \text{otherwise.}
\end{cases}
\label{eq:epsi-2}
\end{equation}
The verification for $t\leq n/s$ as the induction base case. That
means, our assumption needs to depend on the block size $s$ as well:
\begin{assum} We assume the learning rate $\alpha$ is chosen such that $\kappa(s) < 1$ and 
$$ \kappa(s) + \alpha \mu(s) (n/s-1)<1.$$
    \label{assum:alpha-2}
\end{assum}

To interpret this, it is not hard to observe that $\kappa(s)\leq\kappa_{1}$
where $\kappa_{1}$ is computed with a different $\alpha_{1}=\alpha/s^{4}$
since for $i=1,...n/s$ 
\begin{align*}
\|\mathcal{I}-\alpha\tilde{\mathcal{A}}_{i}^{sT}*\tilde{\mathcal{A}}_{i}^{s}\|_{op} & =\left\Vert \mathcal{I}-\alpha\left(\tilde{\mathcal{A}}_{i\cdot(s-1)+1}+\tilde{\mathcal{A}}_{i\cdot(s-1)+2}+\cdots+\tilde{\mathcal{A}}_{i\cdot s}\right)^{T}*\right. \\
& \left. \left(\tilde{\mathcal{A}}_{i\cdot(s-1)+1}+\tilde{\mathcal{A}}_{i\cdot(s-1)+2}+\cdots+\tilde{\mathcal{A}}_{i\cdot s}\right)\right\Vert _{op}\\
 & =\left\Vert \mathcal{I}-\alpha\sum_{i_{1},i_{2}=i\cdot(s-1)+1,\cdots,i\cdot s}\tilde{\mathcal{A}}_{i_{1}}^{T}*\tilde{\mathcal{A}}_{i_{2}}\right\Vert _{op}\\
 & \geq\left|\left\Vert \mathcal{I}\right\Vert _{op}-\alpha\left\Vert \sum_{i_{1},i_{2}=i\cdot(s-1)+1,\cdots,i\cdot s}\tilde{\mathcal{A}}_{i_{1}}^{T}*\tilde{\mathcal{A}}_{i_{2}}\right\Vert _{op}\right|\\
 & \geq\left|\left\Vert \mathcal{I}\right\Vert _{op}-s^{2}\alpha\max_{i=1,...n}\left\Vert \tilde{\mathcal{A}}_{i}^{T}*\tilde{\mathcal{A}}_{i}\right\Vert _{op}\right|
\end{align*}
The $\mu(s)\leq s^{2}\mu$ holds with a similar calculation. 
\begin{align*}
\|\tilde{\mathcal{A}}_{i}^{sT}*\tilde{\mathcal{A}}_{j}^{s}\|_{op} & =\left\Vert \left(\tilde{\mathcal{A}}_{i\cdot(s-1)+1}+\tilde{\mathcal{A}}_{i\cdot(s-1)+2}+\cdots+\tilde{\mathcal{A}}_{i\cdot s}\right)^{T}*\left(\tilde{\mathcal{A}}_{j\cdot(s-1)+1}+\tilde{\mathcal{A}}_{j\cdot(s-1)+2}+\cdots+\tilde{\mathcal{A}}_{j\cdot s}\right)\right\Vert _{op}\\
 & =\left\Vert \sum_{\begin{array}{c}
i_{1}=i\cdot(s-1)+1,\cdots,i\cdot s\\
i_{2}=j\cdot(s-1)+1,\cdots,j\cdot s
\end{array}}\tilde{\mathcal{A}}_{i_{1}}^{T}*\tilde{\mathcal{A}}_{i_{2}}\right\Vert _{op}\\
 & \leq s^{2}\max_{i\neq j,i,j=1,\cdots,n}\left\Vert \tilde{\mathcal{A}}_{i}^{T}*\tilde{\mathcal{A}}_{j}\right\Vert _{op}.
\end{align*}
These calculation shows that 
\begin{align*}
\kappa(s)+\alpha\mu(s)(n/s-1) & \leq\kappa_{1}+\alpha s^{2}\mu(n/s-1).
\end{align*}
And the Assumption \ref{assum:alpha-2} becomes more restrictive for
larger block size $s$. In practice these are a smaller set ($n/s<n$)
of conditions compared to assumptions in Theorem \ref{thm:conv-cyclic}.

Analogous to the argument in Theorem \ref{thm:conv-cyclic}, via \eqref{eq:iter_bound-2} we have that: 
\begin{align}
\mathcal{E}(t+1) & =\left\Vert \mathcal{X}(t+1)-\mathcal{X}_{*}\right\Vert _{F}\nonumber \\
 & =\left\Vert \mathcal{X}(t)+\alpha\tilde{\mathcal{A}}_{\left[t\Mod n\right]+1}^{T}*\mathcal{R}(t+1)-\mathcal{X}_{*}\right\Vert _{F}\nonumber \\
 & \leq\left\Vert \left(\mathcal{I}-\alpha\tilde{\mathcal{A}}_{\left[t\Mod n\right]+1}^{T}*\tilde{\mathcal{A}}_{\left[t\Mod n\right]+1}\right)*\left(\mathcal{X}(t)-\mathcal{X}_{*}\right)\right\Vert _{F}\nonumber \\
 & \quad\quad+\left\Vert \alpha\tilde{\mathcal{A}}_{\left[t\Mod n\right]+1}^{T}*\left(\sum_{i=1}^{n-1}\tilde{\mathcal{A}}_{\left[(t-i)\Mod n\right]+1}*\left(\mathcal{X}_{*}-\mathcal{X}(t-i)\right)\right)\right\Vert _{F}\nonumber \\
 & \leq\left\Vert \mathcal{I}-\alpha\tilde{\mathcal{A}}_{\left[t\Mod n\right]+1}^{T}*\tilde{\mathcal{A}}_{\left[t\Mod n\right]+1}\right\Vert _{op}\left\Vert \left(\mathcal{X}_{*}-\mathcal{X}(t)\right)\right\Vert _{F}\nonumber \\
 & \quad\quad+\alpha\sum_{i=1}^{n-1}\left\Vert \tilde{\mathcal{A}}_{\left[t\Mod n\right]+1}^{T}*\tilde{\mathcal{A}}_{\left[(t-i)\Mod n\right]+1}\right\Vert _{op}\left\Vert \left(\mathcal{X}_{*}-\mathcal{X}(t-i)\right)\right\Vert _{F}\nonumber \\
 & =\left\Vert \mathcal{I}-\alpha\tilde{\mathcal{A}}_{\left[t\Mod n\right]+1}^{T}*\tilde{\mathcal{A}}_{\left[t\Mod n\right]+1}\right\Vert _{op}\mathcal{E}(t)\nonumber \\
 & +\alpha\sum_{i=1}^{n-1}\left\Vert \tilde{\mathcal{A}}_{\left[t\Mod n\right]+1}^{T}*\tilde{\mathcal{A}}_{\left[(t-i)\Mod n\right]+1}\right\Vert _{op}\mathcal{E}(t-i)\nonumber \\
 & \leq\kappa\mathcal{E}(t)+\alpha\mu\sum_{i=1}^{n-1}\mathcal{E}(t-i),\label{eq:iter_bound-1}
\end{align}
\begin{align*}
\mathcal{E}(t+1) & \leq\kappa(s)^{t}\epsilon_{T}(s)+\alpha\mu(s)\sum_{j=0}^{t}\sum_{i=1}^{n/s}\kappa(s)^{j}\mathcal{E}(t-j-i),\\
 & =\left\Vert \mathcal{X}(t+1)-\mathcal{X}_{*}\right\Vert _{F}\\
 & =\left\Vert \mathcal{X}(t)+\alpha\tilde{\mathcal{A}^{sT}}_{\left[t\Mod n/s\right]+1}*\mathcal{R}(t+1)-\mathcal{X}_{*}\right\Vert _{F}\\
 & \leq\left\Vert \left(\mathcal{I}-\alpha\tilde{\mathcal{A}^{sT}}_{\left[t\Mod n/s\right]+1}*\tilde{\mathcal{A}}_{\left[t\Mod n/s\right]+1}^{s}\right)*\left(\mathcal{X}(t)-\mathcal{X}_{*}\right)\right\Vert _{F}\\
 & \quad\quad+\left\Vert \alpha\tilde{\mathcal{A}^{sT}}_{\left[t\Mod n/s\right]+1}*\left(\sum_{i=1}^{n/s-1}\tilde{\mathcal{A}}_{\left[(t-i)\Mod n/s\right]+1}*\left(\mathcal{X}_{*}-\mathcal{X}(t-i)\right)\right)\right\Vert _{F}\\
 & \leq\left\Vert \mathcal{I}-\alpha\tilde{\mathcal{A}^{sT}}_{\left[t\Mod n/s\right]+1}*\tilde{\mathcal{A}}_{\left[t\Mod n/s\right]+1}^{s}\right\Vert _{op}\left\Vert \left(\mathcal{X}_{*}-\mathcal{X}(t)\right)\right\Vert _{F}\\
  \end{align*}
 \begin{align*}
 & \quad\quad+\alpha\sum_{i=1}^{n/s-1}\left\Vert \tilde{\mathcal{A}^{sT}}_{\left[t\Mod n/s\right]+1}*\tilde{\mathcal{A}}_{\left[(t-i)\Mod n/s\right]+1}^{s}\right\Vert _{op}\left\Vert \left(\mathcal{X}_{*}-\mathcal{X}(t-i)\right)\right\Vert _{F}\\
 & =\left\Vert \mathcal{I}-\alpha\tilde{\mathcal{A}^{sT}}_{\left[t\Mod n/s\right]+1}*\tilde{\mathcal{A}}_{\left[t\Mod n/s\right]+1}^{s}\right\Vert _{op}\mathcal{E}(t)+\\&\alpha\sum_{i=1}^{n/s-1}\left\Vert \tilde{\mathcal{A}^{sT}}_{\left[t\Mod n/s\right]+1}*\tilde{\mathcal{A}}_{\left[(t-i)\Mod n/s\right]+1}^{s}\right\Vert _{op}\mathcal{E}(t-i)\\
 & \leq\kappa(s)\mathcal{E}(t)+\alpha\mu(s)\sum_{i=1}^{n/s}\mathcal{E}(t-i)\leq...\\
 & \leq\left(\kappa(s)^{t+1}+\alpha\mu(s)\sum_{j=0}^{t}\kappa(s)^{j}\sum_{i=1}^{n/s-1}\epsilon_{t-j-i}(s)\right)\mathcal{E}(0)\\
 & =\epsilon_{T+1}(s)\mathcal{E}(0).
\end{align*}
With the Assumption, we can show 
\begin{align*}
\epsilon_{(T+1)+1}(s) & =\kappa^{(T+1)+1}(s)+\alpha\mu(s)\sum_{j=0}^{T+1}\kappa(s)^{j}\sum_{i=1}^{n/s-1}\epsilon_{T+1-j-i}(s)\\
 & =\kappa^{(T+1)+1}(s)+\alpha\mu(s)\sum_{j=0}^{T}\kappa(s)^{j}\sum_{i=1}^{n/s-1}\epsilon_{T+1-j-i}(s)+\alpha\mu(s)\kappa(s)^{T+1}\sum_{i=1}^{n/s-1}\epsilon_{T+1-(T+1)-i}(s)\\
 & =\kappa^{(T+1)}(s)\left(\kappa(s)+\alpha\mu(s)\cdot(n/s-1)\right)+\sum_{j=0}^{T}\kappa(s)^{j}\sum_{i=1}^{n/s-1}\epsilon_{T+1-j-i}(s)\\
 & <\kappa^{(T+1)}(s)\left(\kappa(s)+\alpha\mu(s)\cdot(n/s-1)\right)+\sum_{j=0}^{T}\kappa(s)^{j}\sum_{i=1}^{n/s-1}\epsilon_{T-j-i}(s)\\
 & <\kappa^{(T+1)}(s)+\alpha\mu(s)\cdot\sum_{j=0}^{T}\kappa(s)^{j}\sum_{i=1}^{n/s-1}\epsilon_{T-j-i}(s)=\epsilon_{T+1}(s).
\end{align*}

This also allows us to discuss a special case, where the block size
$s=n$, and the Algorithm \ref{alg:tsolve_cyclic_slices} reduces
to full gradient descent for \eqref{eq:tensorlinsys}. Then, our Assumption \ref{assum:alpha-2}
becomes this only sufficient condition $\kappa<1$, which means that unless all slices are orthogonal under t-product, the full gradient descent should converge.

%
%
%
%
%
%
%
%
%
%

\end{document}